\documentclass{amsart}

\makeatletter
 \def\LaTeX{\leavevmode L\raise.42ex
   \hbox{\kern-.3em\size{\sf@size}{0pt}\selectfont A}\kern-.15em\TeX}
\makeatother

\newcommand{\BibTeX}{{\rm B\kern-.05em{\sc
i\kern-.025emb}\kern-.08em\TeX}}
\newtheorem{col}{Corollary}[section]
\newtheorem{theorem}{Theorem}[section]
\newtheorem{lemma}[theorem]{Lemma}
\newtheorem{remark}[theorem]{Remark}
\newtheorem{definition}[theorem]{Definition}
\theoremstyle{definition}

\numberwithin{equation}{section}

\newcommand{\mfd}[1]{\mathbf{#1}}

\newcommand\Hilb{\mathcal H}
\newcommand\PWo{{PW_{\negthinspace \omega}}}
\newcommand\PWoL{{PW_{\negthinspace \omega}(\sqrt{L})}}
\def\PWo{{{\bf PW}_{\negthinspace \omega}}}
\def\PWoL{{{\bf PW}_{\negthinspace \omega}(\sqrt{L})}}
\newcommand\FSL{{F(\sqrt{L})}}
\newcommand\FtSL{{F(t \sqrt{L})}}  % <<< CHANGED HERE
\newcommand\SLB{{(\sqrt{L})}}

\newcommand\McL{{\mathcal{L}}}
\newcommand\ScLB{{(\sqrt{\McL})}}
\newcommand\bPW{{ {\bf PW}}}
\newcommand\muk{{a}}

\def\EB{{\bf E}}
\def\FB{{\bf F}}
\def\AB{{\bf A}}
\def\HB{{\bf H}}

\begin{document}

% \title[Sampling and frames in non-Euclidean settings]{Shannon sampling and space-frequency localized frames in non-Euclidean settings }

\title[Geometric Space-Frequency Analysis]
{Geometric Space-Frequency Analysis on Manifolds}

%\begin{center}
\author {Hans G. Feichtinger}
\address{Faculty of Mathematics, NuHAG, University  Vienna,
AUSTRIA; \\ hans.feichtinger@univie.ac.at}

\author{Hartmut F\"uhr}\address{
Lehrstuhl A f\"ur Mathematik, RWTH Aachen, % D-52056 Aachen,
GERMANY;  fuehr@matha.rwth-aachen.de}

\author{Isaac Z. Pesenson }\address{ Department of Mathematics,
Temple University,  Philadelphia, PA 19122;
\\ pesenson@temple.edu. }

%\end{center}

\begin{abstract}

 This paper gives a survey of methods for the construction of space-frequency concentrated frames on Riemannian manifolds with bounded curvature, and the applications of these frames to the analysis of function spaces. In this general context, the notion of frequency is defined using the spectrum of a distinguished differential operator on the manifold, typically the Laplace-Beltrami operator. Our exposition starts with the case of the real line, which serves as motivation and blueprint for the material in the subsequent sections. 
 
 After the discussion of the real line, our presentation starts out in the most abstract setting proving rather general sampling-type results for appropriately defined  Paley-Wiener vectors in  Hilbert spaces. 
These results allow a handy construction of  Paley-Wiener frames in $L_2(\mfd{M})$, for a Riemann manifold of bounded geometry, essentially by taking a partition of unity in frequency domain.  The discretization of the associated integral kernels then gives rise to frames consisting of smooth functions in $L_2(\mfd{M})$, with fast decay in space and frequency. These frames are used to introduce new norms in  corresponding Besov spaces on $\mfd{M}$.
  
For compact Riemannian manifolds the theory extends to $L_p$ and  Besov spaces.  Moreover, for compact homogeneous manifolds, one obtains the so-called product property for eigenfunctions of certain operators  and proves  a cubature formulae with positive coefficients which allow to construct Parseval frames that characterize  Besov spaces in terms of coefficient decay. 
  
Throughout the paper, the general theory is exemplified with the help of various concrete and relevant examples, such as the unit sphere and the Poincar\'{e} half plane. 
\end{abstract}

\maketitle

\tableofcontents

\section{Introduction}

In 2004 (see   
\cite{FP04}) H. Feichtinger and  I.  Pesenson  wrote:
\textit{"It is our strong belief that there  exist many real life problems
in Signal Analysis and Information Theory  which would require
non-Euclidean models.
A Theory which will unify the ideas of the Classical
Sampling Theorem, one of the most beautiful and applied results of
the Euclidean Fourier Analysis, with ideas of Differential
Geometry and non-Euclidean Harmonic Analysis would be of great
interest and importance.
We consider the present paper as a foundation for future papers in
which we are planning to investigate in details more specific
examples such as:
1) Spheres, projective spaces and general compact manifolds.
2) Hyperboloids and general non-compact symmetric spaces.
3) Various  Lie groups.}" 

The purpose of this survey is to provide an
introduction to emerging theories  of Shannon-type sampling and space-frequency localized frames in various non-Euclidean settings.

We report on  the
Shannon sampling theory, approximation theory, space-frequency localized  frames, and Besov spaces on compact and non-compact manifolds which were developed in  \cite{FP04}, \cite{FP05},  \cite{gp}-\cite{gpa}, \cite{Pes79}-\cite{Pes16}.  	These topics are not only of a theoretical interest.
Many important applications of multiresolution analysis on manifolds were developed for   imaging,  geodesy,  cosmology, crystallography, scattering theory,  biology,   and statistics (see  \cite{BKMP1},  \cite{BKMP2}, \cite{BEP}, \cite{BP}, \cite{CM}, \cite{DFHMP},   \cite{EFM}, \cite{FGS}, \cite{FV}, \cite{GM100},  \cite{K}-\cite{KCM}, \cite{KR}, \cite{M2}-\cite{M-all}, \cite{P}).

We begin with introducing a rather general, spectral-theoretic setup that allows to prove Shannon-type sampling theorems in abstract Hilbert spaces, as well as the definition and characterization of Besov-type spaces, in a unified language. We then use these results to study sampling theorems, the construction of Paley-Wiener (bandlimited) frames and the characterization of function spaces on Riemannian manifolds (see \cite{Pes00}, \cite{Pes01}, \cite{Pes07}, \cite{Pes09b}). The approach works for rather wide  classes of Riemannian manifolds, such as general compact manifolds without boundary, bounded domains with smooth boundaries in Euclidean spaces, or non-compact Riemannian manifolds of bounded geometry whose Ricci curvature is bounded from below.

For  {\em compact} Riemannian manifolds we
prove generalizations of the Bernstein, Bernstein-Nikolskii and Jackson inequalities.  
In the case of compact manifolds we go beyond the purely Hilbert space theoretic setting, and  include $L_p$-spaces in the discussion as well, for $1 \le p \le \infty$. This allows to characterize elements of the Besov spaces $\mathcal{B}^{\alpha}_{p,q}(\mfd{M})$ in terms of approximations by eigenfunctions of elliptic differential operators on $\mfd{M}$. For the case of a compact {\em homogeneous} manifold $\mfd{M}$
 we further sharpen the Bernstein and  Bernstein-Nikolskii
 inequalities, using global derivatives with respect to
 specific vector fields on $\mfd{M}$; here the above-mentioned elliptic differential operator is the Casimir operator $\McL$. Furthermore, we construct Parseval  bandlimited and localized frames in $L_{2}(\mfd{M})$ for this setting, and show that they serve to characterize Besov spaces via coefficient decay.

\subsection{Overview of the paper} In section \ref{sec2} we discuss three ways of constructing Paley-Wiener-Schwartz frames  in $L_{2}(\mathbb{R})$. In \ref{subsec2.1} we are using the Fourier transform to introduce functions of the non-negative square root $\sqrt{-d^{2}/dx^{2}}$ in the space $L_{2}(\mathbb{R})$. In \ref{Firstmethod} we use these results to construct what we call nearly Parseval Paley-Wiener-Schwartz frames in $L_{2}(\mathbb{R})$ which are comprised of functions which are bandlimited and have fast decay at infinity. In subsection \ref{Secondmethod} we explore the classical Sampling Theorem to construct Parseval  Paley-Wiener-Schwartz frames. In subsection \ref{Thirdmethod} we establish a cubature formula with positive coefficients for functions in Paley-Wiener space. By means of such formulas and the fact that product of two Paley-Wiener functions is another Paley-Wiener function we develop a third method of constructing Parseval Paley-Wiener-Schwartz frames in $L_{2}(\mathbb{R})$.  

It is the objective of the present article to show how the ideas and methods illustrated in \ref{sec2}  can be extended to Riemannian manifolds.

    In our paper we make systematic use of the Spectral Theorem for self-adjoint operators as a substitute for the classical Fourier transform. The approach is motivated by  examples described in section \ref{sec2}. However, it should be noted that although the regular Fourier transform considered in section \ref{sec2} provides spectral resolution for the operator of the first derivative  $-id/dx$ in $L_{2}(\mathbb{R})$ but it  is not a spectral resolution for the operator $-d^{2}/dx^{2}$ and its non-negative  root $\sqrt{-d^{2}/dx^{2}}$ in $L_{2}(\mathbb{R})$.
      
  In Section \ref{Hilbert}  we introduce a   notion of Paley-Wiener (bandlimited) vectors in a Hilbert space $\Hilb$ which is equipped with a self-adjoint operator $D$, and develop a Shannon-type sampling of such vectors. By constructing appropriate projections of $\Hilb$  onto subspaces of  $\omega$-Paley-Wiener vectors
$\PWo(D),\>\omega>0$
we construct bandlimited  frames in $\Hilb$.
 We define Besov spaces as interpolation spaces
 between $\Hilb$ and domains of $D^{k},\>k\in \mathbb{N}$, and
 show that they can be described in terms of frame coefficients.

Our approach in \cite{Pes01} was to treat a set of "samples" of a vector $f\in H$ as a set of values $\psi_{\nu}(f)$  for a specific "sampling" family of functionals $\psi_{\nu}$ for which Plancherel-Polya-type inequalities (=frame inequalities) hold on  Paley-Wiener subspaces. The Spectral Theorem allows to decompose every vector in $\mathcal{H}$ into a series of Paley-Wiener vectors. Then an application of our Sampling Theorem \ref{Ss} leads to the construction of  Paley-Wiener frames for  $\mathcal{H}$ (Theorem \ref{frameH}).

 In Subsection \ref{Interp} we formulate and prove  an  important result (Theorem \ref{equivalence-interpolation})  about interpolation and approximation spaces. This result is essentially due to  Peetre-Sparr \cite{PS} and Butzer-Scherer \cite{BSch},
but we formulate and prove it in a form which is most suitable for our purposes (see also \cite{KP}). In particular, our formulation is more general than a similar Theorem 9.1 in Ch. 7 in \cite{DVL}. In Subsection \ref{AbstractBesov} we describe abstract Besov subspaces in terms of approximations by Paley-Wiener vectors and in terms of coefficients with respect to our Paley-Wiener frames.

When it comes to the space $\Hilb =L_{2}(\mfd{M})$, where $\mfd{M}$ is a manifold, the families of "sampling" functionals $\{\psi_{\nu}\}$  are just   families of compactly supported distributions (with small supports) associated with what we call metric lattices of points $\{x_{k}\}$ on $\mfd{M}$. This term is used to emphasize that the points  $\{x_{k}\}$ are distributed over $\mfd{M}$
``almost uniformly'' and that they are separated.

Not every metric space
possesses  lattices of points  with the properties we need for our sampling theory. In  Subsection \ref{sub-1} we clarify this issue. It was shown in  \cite{Pes04b}  that if a Riemannian manifold has bounded geometry and its Ricci curvature is bounded from below   then one can construct a sequence of lattices whose mesh radius tends to zero.  Note that  the property of bounded geometry is essentially equivalent to the fact  that  all covariant derivatives of the Riemann curvature are bounded from above. This shows that our conditions are rather natural, since appropriate uniformly distributed and separated sets of points  can exist only if the curvature (in one sense or another)
is bounded from above and from below. The rest of the Section \ref{Manifolds} is devoted to descriptions of manifolds of common interest.

In Section \ref{SFonM} we implement  the general scheme of Section 2 in the spaces $L_{2}(\mfd{M})$, for manifolds $\mfd{M}$ satisfying the assumptions of Section 3. Following the general scheme of Section \ref{Hilbert}, the central tool for sampling theory are Poincar\'e-type estimates, which we derive for Riemannian manifolds in Section 4.   To construct frames which are almost tight we use the so-called {\it average } sampling in a way that includes and generalizes the pointwise sampling.  In the case of a straight line average sampling was considered for the first time in \cite{FG}. In the case of manifolds average sampling was developed in \cite{Pes04b}, \cite{Pes06}. We do not discuss reconstruction algorithms in detail, but it can be done by the following methods (besides using dual frames):
(1) reconstruction by using variational splines on manifolds  \cite{Pes98a}-\cite{Pes09b};$\>\>\>$
 (2) reconstruction using iterations  \cite{FP04}, \cite{FP05};$\>\>\>$
 (3) the frame algorithm \cite{Gr}.

 In Sections \ref{kernels}  we introduce and analyze kernels
 associated to elliptic differential operators on general compact Riemannian manifolds. The most important result here concerns the localization of such kernels (Theorem \ref{kernelsize}) which implies an analogue of the Littlewood-Paley decomposition of functions in the spaces $L_{p}(\mfd{M}), 1\leq p\leq \infty,$ on compact Riemannian manifolds (Theorem \ref{LPT}).  One can find other approaches  to the Littlewood-Paley decompositions 
on manifolds (see for example \cite{KR},  \cite{Sog}, \cite{SSS}).

 Section \ref{PHM} is devoted to Parseval space-frequency localized frames on compact homogeneous manifolds. The main result here  is  Theorem \ref{Pfhm} which was proved in \cite{gp}. This Theorem is based on two non-trivial facts: 
 the  product of eigenfunctions of certain elliptic  differential operators (Theorem \ref{prodthm}) and on a cubature formula with positive coefficients which allows for exact integration of respected eigenfunctions (Theorem \ref{cubformula}). 
Let us just mention  that a set of similar facts holds true for sub-Laplacians  on compact homogeneous manifolds \cite{Pes15}.

 In Section \ref{AppSec} an approximation theory by eigenfunctions of elliptic  operators on compact manifolds is developed. These results lead to a characterization of Besov spaces in terms of sampling (Theorem \ref{BesSampl}). In Section \ref{Apphm} we discuss approximation theory on compact homogeneous manifolds. In particular, a mixed modulus of continuity is introduced and Besov spaces are characterized in terms of this modulus of continuity \cite{Pes79}-\cite{Pes88a}.  Section \ref{BesFr} contains characterization of Besov spaces in terms of frame coefficients. 
Similar theorems can be proved in the case of a sub-Laplacian and sub-elliptic spaces on compact homogeneous manifolds \cite{Pes15}.

Here is a very brief account of related work, mostly by other authors. The papers \cite{BG}, \cite{SS},  \cite{BDY},   \cite{CM}-\cite{DST},  \cite{FM1}, \cite{HF-1}, \cite{FuhrM},   \cite{FMV}, \cite{HMYa},  \cite{HMYb}, \cite{KPet}, \cite{MM}, \cite{MY},  \cite{NPW1},  \cite{NPW2}, \cite{YZ}, contain a number of results about  frames,  wavelets, and Besov spaces on  Riemannian manifolds, on Lie groups  of polynomial growth,   on  metric-measure spaces, and on quasi-metric measure spaces.  One can say that most of  these papers generalize and further develop  ideas which are rooted in the classical  Littlewood-Paley theory and/or Calderon reproducing formula.

In particular, it is well understood by now that a productive  generalization of the  Littlewood-Paley theory should be based on a decomposition of identity operator into a series of kernel operators with appropriately localized kernels. It   was proved in  \cite{HMYa},  \cite{HMYb}   that any reasonably nice metric measure space admits such decomposition.

Among the papers devoted to metric-measure and quasi-metric measure spaces the articles  \cite{CM},  \cite{CKP}, \cite{FM1},   \cite{KPet}, \cite{MM},   \cite{NPW1} are the closest  to our approach, since they also aim to construct space-frequency localized frames. To incorporate a notion of frequency into a setting of quasi-metric measure spaces, the authors of these papers  impose some additional conditions.  They, essentially, assume the existence of a self-adjoint operator in a corresponding $L_{2}$ space  whose  heat semigroup is a kernel operator with a kernel  obeying estimates resembling the heat kernel estimates on Euclidean space. This way they are able to define a notion of bandlimitedness  and to construct space-frequency localized frames.
The most advanced results in such setting were recently  obtained in  \cite{CKP}.

It should be noted that the most interesting situations
where one finds the conditions of these papers satisfied
%are in fact  metric-measure spaces
are still manifolds:
compact Riemannian manifolds,  non-compact manifolds
with curvature bounded from below, groups of polynomial
grows and their homogeneous manifolds.

Again, our goal is to give a concise introduction to the  fast developing subject of sampling and wavelet-like frames on manifolds
by a description of a few  underlying ideas,
which provide a basis for 
 discoveries in \cite{FP04}, \cite{FP05},  \cite{gp}-\cite{gpa}, \cite{Pes79}-\cite{Pes16}.

\section{Basic Example: Paley-Wiener-Schwartz frames  in $L_{2}(\mathbb{R})$}\label{sec2}

\subsection{Smooth decomposition of $L_{2}(\mathbb{R})$ into Paley-Wiener subspaces}\label{subsec2.1}

We take the first-order pseudo-differential operator $D=\sqrt{-d^{2}/dx^{2}}$ as the positive square root of the positive operator $-d^{2}/dx^{2}$. If $F$ belongs to  the Schwartz  space $\mathcal{S}(\mathbb{R})$  then following the spirit of the Spectral Theorem (see (\ref{Op-function}))  one can introduce
 the operator $F(D)$ by the formula
\begin{equation}\label{op-func}
F\left(\sqrt{-d^{2}/dx^{2}}\right)f(x)=\frac{1}{\sqrt{2\pi}}\int_{\mathbb{R}}e^{ix\lambda}F(\lambda)\widehat{f}(\lambda)d\lambda,\>\>\>f\in  \mathcal{S}(\mathbb{R}),
\end{equation}
where the Fourier transform  $\widehat{f}$ is defined as
$$
\widehat{f}(\lambda)=\frac{1}{\sqrt{2\pi}}\int_{\mathbb{R}}e^{-i x\lambda }f(x)dx,\>\>\>f\in \mathcal{S}(\mathbb{R}).
$$
The operator $F\left(\sqrt{-d^{2}/dx^{2}}\right)$ is  convolution  with the Schwartz function $\check{F} \in  \mathcal{S}(\mathbb{R})$ which is  inverse Fourier transform of $F$:
$$
F\left(\sqrt{-d^{2}/dx^{2}}\right)f(x)=\frac{1}{\sqrt{2\pi}}\int_{\mathbb{R}}\check{F}(x-y)f(y)dy.
$$
 In particular, for any positive $t$ we have
 \begin{equation}\label{kernel11}
F\left(t\sqrt{-d^{2}/dx^{2}}\right)f(x)=\frac{1}{\sqrt{2\pi}}\int_{\mathbb{R}}\frac{1}{t}\check{F}\left(\frac{x-y}{t}\right)f(y)dy=\int_{\mathbb{R}}K^{F}_{t}(x,y)f(y)dy,
\end{equation}
where
\begin{equation}\label{kernel12}
K^{F}_{t}(x,y)=\frac{1}{t\sqrt{2\pi}}\check{F}\left(\frac{x-y}{t}\right).
\end{equation}
Moreover, one clearly has that for any $N>0$
there exists a constant $C_{N}$ such that
\begin{equation}\label{S-kernel-estim}
\left|K^{F}_{t}(x,y)\right|\leq \frac{C_{N}}{t}{\left[1+\frac{|x-y|}{t}\right]}^{-N}.
\end{equation}

 Let $g\in C^{\infty}(\mathbb{R})$ be a non-increasing
%  monotonic
 function such that $supp(g)\subset [-2,\>  2], $ and $g(\lambda)=1$ for $\lambda\in [-1,\>1], \>0\leq g(\lambda)\leq 1.$
 We now let
\[
 h(\lambda) = g(\lambda) - g(2 \lambda)~,
\] which entails $supp(h) \subset [-2,-2^{-1}] \cup   [2^{-1},2]$,
and use this to define
\begin{equation} \label{functions2}
 F_0(\lambda) = \sqrt{g(\lambda)}~, F_j(\lambda) = \sqrt{h(2^{-j} \lambda)}~, j \ge 1~,\>j\in \mathbb{N},
\end{equation}
as well as
\[
 G_j(\lambda) = \left[F_j(\lambda)\right]^2=F_j^2(\lambda)~, j \ge 0~,\>j\in \mathbb{N}.
\]
As a result of the definitions, we get for all $\lambda \in \mathbb{R}$ the equations
\begin{equation} \label{partsums3}
\sum_{j = 0}^m G_j(\lambda) = \sum_{j = 0}^m F_j^2(\lambda)
=  g(2^{-m} \lambda),
\end{equation}
and as a consequence
\begin{equation} \label{partsums4}
\sum_{j \ge 0} G_j(\lambda) = \sum_{j \ge 0} F_j^2(\lambda) =  1~,\>\>\>\lambda\in \mathbb{R},
\end{equation}
 with finitely many nonzero terms occurring in the sums for each
 fixed $\lambda$. %uuu
One has
 % {\sc here SLB}
$$
F_{j}^{2}\left(\sqrt{-d^{2}/dx^{2}}\right)  f=\mathcal{F}^{-1}\left(F_{j}^{2}(\lambda)\mathcal{F}f(\lambda)\right),\>\>\>j\geq 1,\>\>j\in \mathbb{N},
$$
and thus
\begin{equation} \label{eqn:quad_part_identity_1}
 f = \mathcal{F}^{-1}\mathcal{F}f(\lambda) =\mathcal{F}^{-1}\left(\sum_{j\geq 0}F_{j}^{2}(\lambda)\mathcal{F}f(\lambda)\right) =
 $$
 $$
  \sum_{j\geq 0} F_{j}^2\left(\sqrt{-d^{2}/dx^{2}}\right)f= \sum_{j \geq 0} G_j\left(\sqrt{-d^{2}/dx^{2}}\right) f . 
\end{equation}  % {\sc here SLB}
Since $F_{j}\left(\sqrt{-d^{2}/dx^{2}}\right)$ is a self-adjoint operator we obtain
$$
\left \|F_{j}\left(\sqrt{-d^{2}/dx^{2}}\right) f\right \|^{2}=\left \langle F_{j}\left(\sqrt{-d^{2}/dx^{2}}\right) f, \left(\sqrt{-d^{2}/dx^{2}}\right) f \right\rangle=
$$
$$
\left \langle F_{j}^{2}\left(\sqrt{-d^{2}/dx^{2}}\right) f, f \right\rangle
$$  % {\sc here SLB}
and then
\begin{equation}
\label{norm equality-00}
\|f\|^2=\sum_{j\geq 0}\left \langle F_{j}^2\left(\sqrt{-d^{2}/dx^{2}}\right) f,f\right\rangle=\sum_{j\geq 0}\left \|F_{j}\left(\sqrt{-d^{2}/dx^{2}}\right) f\right \|^2 .
\end{equation}
Since the functions $G_j,  F_{j}$, have their supports in  $
[-2^{j+1},\>\>-2^{j-1}]\cup[2^{j-1},\>\>2^{j+1}]$, the functions $ F_{j}^{2}\left(\sqrt{-d^{2}/dx^{2}}\right) f $ and $G_j\left(\sqrt{-d^{2}/dx^{2}}\right) f$
 are bandlimited to  $[-2^{j+1},\>\>-2^{j-1}]\cup[2^{j-1},\>\>2^{j+1}]$, whenever $j \ge 1$, and to $[-2,2]$ for $j=0$. They clearly belong to  the Schwartz  space $\mathcal{S}(\mathbb{R})$.

\subsection{A method of constructing almost Parseval Paley-Wiener-Schwartz frames}\label{Firstmethod}

\begin{definition}
The Paley-Wiener space $PW_{\omega}(\mathbb{R}),\>\>\omega>0,$ is introduced as the space of all $f\in L_{2}(\mathbb{R})$ whose $L_{2}$-Fourier transform has support in $[-\omega, \>\omega]$.

\end{definition}

Using the Fourier transform one can easily verify that 
a function $f\in L_{2}(\mathbb{R})$ belongs to the space $PW_{\omega}(\mathbb{R})$ if and only if the following Bernstein inequality holds
$$
\|f^{(r)}\|\leq \omega ^{r}\|f\|\>\>\>r\in \mathbb{N}.
$$

For a given $\rho>0$ and $0<\epsilon<1$ consider a sequence $\{x_{k}\}$ such that 
\begin{equation}\label{partition}
|x_{k}-x_{k+1}|\leq \rho,\>\>\>\>|x_{k}-x_{k+1}|\geq \rho/(1+\epsilon)
\end{equation}
and set $I_{k}=(x_{k}, \>x_{k+1})$. The Fundamental Theorem of calculus  and the Holder inequality  imply
$$
\int_{x_{k}}^{x}f^{'}(t)dt=f(x)-f(x_{k}),\>\>x_{k}\leq x\leq x_{k+1}
$$

$$
|f(x)-f(x_{k})|\leq \int_{x_{k}}^{x}|f^{'}(t)|dt\leq \rho^{1/2}\left(\int_{I_{k}}|f^{'}|^{2}\right)^{1/2},
$$
$$
|f(x)-f(x_{k})|^{2}\leq\rho\int_{I_{k}}|f^{'}|^{2},
$$
and another integration over $I_{k}$ gives
\begin{equation}\label{FTC}
\|f-f(x_{k})\|^{2}_{I_{k}}\leq \rho^{2}\|f^{'}\|^{2}_{I_{k}}.
\end{equation}
For  $0<\alpha<1$ and any $A, B$ one has 
\begin{equation}\label{ineq-11}
(1-\alpha)|A|^{2}\leq \frac{1}{\alpha}|A-B|^{2}+|B|^{2},\>\>0<\alpha<1.
\end{equation}
Using (\ref{FTC}), (\ref{ineq-11}) 
we obtain
\begin{equation}\label{interm-ineq}
\left(1-\frac{\epsilon}{3}\right)\|f\|^{2}\leq\sum_{k}|I_{k}||f(x_{k})|^{2}+\frac{3}{\epsilon}\rho^{2}\|f^{'}\|^{2}.
\end{equation}
Applying the Bernstein inequality for $f\in PW_{\omega}(\mathbb{R})$ we get 
\begin{equation}\label{right-side}
\left(1-\frac{\epsilon}{3}\right)\|f\|^{2}\leq \sum_{k}|I_{k}||f(x_{k})|^{2}+\frac{3}{\epsilon}\rho^{2}\omega\|f\|^{2}.
\end{equation}
Assuming $\omega>1$ and  choosing 
$$
\rho\leq\frac{1}{3}\omega^{-1}\epsilon\leq \frac{1}{3}\omega^{-1/2}\epsilon
$$
we obtain
$$
\frac{3}{\epsilon}\rho^{2}\omega\leq \frac{\epsilon}{3}.
$$
Finally it gives us
\begin{equation}\label{right-hand-side}
\left(1-\frac{2}{3}\epsilon\right)\|f\|^{2}\leq \sum_{k}|I_{k}||f(x_{k})|^{2}.
\end{equation}
We consider the function
\[  \xi(x) = 
\begin{cases}
e \exp(1/(|x|^{2}-1))& \mbox{ if } |x|<1\\
0 & \mbox{ if } |x|\geq 1.
\end{cases}
\]
Set $\xi_{k}(x)=
\xi((1+\epsilon)|I_{k}|^{-1}(x-x_{k})).$ Since 
$$
\xi_{k}(x_{k+1})f(x_{k+1})=0,\>\>\>\>\> \xi_{k}^{'}(x)= (1+\epsilon)|I_{k}|^{-1}\xi((1+\epsilon)|I_{k}|^{-1}(x-x_{k})),
$$
and $|I_{k}|^{-1}\leq (1+\epsilon)\rho^{-1}$  we obtain 
$$
|f(x_{k})|=
\left| \int_{x_{k}}^{x_{k+1}}\left(\xi_{k}(x)f(x)\right)^{'}    \right|\leq  \int_{x_{k}}^{x_{k+1}}\left|  \xi_{k}^{'}(x) f(x)+\xi_{k}(x)f^{'}(x)               \right| \leq                   
$$
$$
(1+\epsilon)^{2}\rho^{-1}\int_{x_{k}}^{x_{k+1}}|f(x)|+\int_{x_{k}}^{x_{k+1}}|f(x)^{'}|\leq(1+\epsilon)^{2}\rho^{-1/2}\|f\|_{I_{k}}+\rho^{1/2}\|f^{'}\|_{I_{k}}.
$$
Then for $f\in PW_{\omega}(\mathbb{R})$
$$
|f(x_{k})|\leq \left((1+\epsilon)^{2}\rho^{-1/2}+\omega\rho^{1/2}\right)\|f\|_{I_{k}}.
$$
It gives
$$
\sqrt{|I_{k}|}|f(x_{k})|\leq \left((1+\epsilon)^{2}+\omega\rho\right)\|f\|_{I_{k}}.
$$
Because $
\rho\leq\frac{1}{3}\omega^{-1}\epsilon,\>\>0<\epsilon<1,
$
we have
$$
(1+\epsilon)^{2}+\omega\rho\leq 1+\frac{10}{3}\epsilon.
$$
This leads to 
$$
\sqrt{|I_{k}|}|f(x_{k})|\leq \left(1+\frac{10}{3}\epsilon\right)|\|f\|_{I_{k}},
$$
and 
\begin{equation}\label{left-side}
\sum_{k}|I_{k}||f(x_{k})|^{2}\leq \left(1+\frac{10}{3}\epsilon\right)^{2}\|f\|^{2}.
\end{equation}
Using the same notations as above we can now formulate
 the following theorem about irregular sampling.
\begin{theorem}

 If $0<\epsilon<1$ and $
0<\rho\leq\frac{\epsilon}{3}\omega^{-1}
$
then Plancherel-Polya inequalities hold
\begin{equation}\label{EP-P}
\left(1-\frac{2}{3}\epsilon\right)\|f\|^{2}\leq \sum_{k}|I_{k}||f(x_{k})|^{2}\leq \left(1+\frac{10}{3}\epsilon\right)^{2}\|f\|^{2}, \>\>f\in PW_{\omega}(\mathbb{R}).
\end{equation}
 
\end{theorem}
See \cite{Ben}, \cite{Boas}, \cite{DS}, \cite{PP1}, \cite{PP2} for the classical  Plancherel-Polya inequalities.

Note, that if $\delta_{x_{k}}$ is a Dirac distribution where $\{x_{k}\}$ are defined in (\ref{partition}) then  the  Plancherel-Polya inequalities (\ref{EP-P}) mean that projections of $\left\{\sqrt{|I_{k}|}\delta_{x_{k}}\right\}$ onto $PW_{\omega}(\mathbb{R})$ form a frame in this space.

We return to notations of section \ref{subsec2.1}. Thus, the operator $F_{j}\left(\sqrt{-d^{2}/dx^{2}}\right) $ is a projector of $L_{2}(\mathbb{R})$ into $PW_{2^{j+1}}(\mathbb{R})$. For a fixed $0<\epsilon<1$ pick a $\rho$ such that  $
0<\rho\leq\frac{\epsilon}{3}\omega^{-1}
$. Let $\left\{I_{k}\right\}$ be a corresponding partition $I_{k}=(x_{k},\>x_{k+1})$ considered in (\ref{partition}).

In the sense of distributions one has 
$$
\left(F_{j}\left(\sqrt{-d^{2}/dx^{2}}\right) \sqrt{|I_{k}|}\delta_{x_{k}}\right)(x)=\left(\mathcal{F}^{-1}\left(F_{j}\sqrt{|I_{k}|}e^{ix_{k}\cdot}\right)\right)(x),
$$
where the formula $\mathcal{F}\delta_{x_{k}}=e^{ix_{k}\xi}$ was used. At the same time, according to (\ref{kernel11})
\begin{equation}\label{frame00}
\left(F_{j}\left(\sqrt{-d^{2}/dx^{2}}\right) \sqrt{|I_{k}|}\delta_{x_{k}}\right)(x)=\int_{\mathbb{R}}K^{F_{j}}_{1}(x,y) \sqrt{|I_{k}|}\delta_{x_{k}} (y)dy=
$$
$$
K^{F_{j}}_{1}(x_{k},y)=\frac{1}{\sqrt{2\pi}}\check{F_{j}}\left(x_{k}-y\right)=\varphi_{x_{k}}^{j}(y)=\varphi_{k}^{j}(y).
\end{equation}
It  is  obvious that every
$\varphi_{k}^{j}$ belongs to $PW_{2^{j+1}}(\mathbb{R})\cap \>\mathcal{S}(\mathbb{R})$ which means that  it  is perfectly localized on the frequency side and "essentially" localized in space(time).

Since the operator $F_{j}\left(\sqrt{-d^{2}/dx^{2}}\right)$ is self-adjoint one has 
$$
\left<F_{j}\left(\sqrt{-d^{2}/dx^{2}}\right)f, \sqrt{|I_{k}|} \delta_{x_{k}}\right>=\left<f, \varphi_{k}^{j}\right>.
$$
Combining (\ref{norm equality-00}) and (\ref{EP-P}) we obtain that 
 the following frame inequalities hold
\begin{equation}\label{EFI}
\left(1-\frac{2}{3}\epsilon\right)\|f\|^{2}\leq \sum_{j}\sum_{k\in \mathbb{Z}}\left|\left<f, \varphi_{k}^{j}\right>\right|^{2}\leq \left(1+\frac{10}{3}\epsilon\right)^{2}\|f\|^{2}, \>\>\>f\in L_{2}(\mathbb{R}).
\end{equation}

 The classical result of Duffin and Schaeffer \cite{DS} says that the
inequalities (\ref{EFI}) imply the existence of a \textit{dual frame}
$\left\{\Phi^{j}_{k}\right\}$  such that any function $f\in
PW_{2^{j+1}}(\mathbb{R})$ can be reconstructed according to the
following formula
\begin{equation}
f(x)=\sum_{j\in \mathbb{Z}}\sum_{k\in \mathbb{Z}}\left<f, \varphi_{k}^{j}\right>\Phi^{j}_{k}(x).
\end{equation}
It is clear  that each $\Phi^{j}_{k}$ belongs to $PW_{2^{j+1}}(\mathbb{R})$.

\subsection{Constructing Parseval Paley-Wiener-Schwartz frames using Classical Sampling Theorem}\label{Secondmethod}

 The classical sampling theorem says, that if
$f$ is $\omega$-bandlimited then $f$ is completely determined by
its values at points $k\pi/\omega, k\in \mathbb{Z}$, and can be
reconstructed in a stable way from the samples $f(k\pi/\omega)$,
i.e.
 \begin{equation}
 f(x)= \sum_{k\in \mathbb{Z}} f\left(k\pi/\omega\right)\frac{\sin(\omega
(x-k\pi/\omega))}{\omega (x-k\pi/\omega )},
\end{equation}
where convergence is understood in the  $L_{2}$-sense. Moreover,
the following equality between "continuous" and "discrete" norms
holds true
\begin{equation}\label{Sh-equality}
\left(\int_{-\infty}^{+\infty}|f(x)|^{2}dx\right)^{1/2}=\left(\sum_{j\in
\mathbb{Z}}\frac{1}{\omega}\left|f(k\pi/\omega)\right|^{2}\right)^{1/2}.
 \end{equation}

This equality follows from the fact that the functions $e^{2\pi i
t(k\pi/\omega)}$ form an orthonormal basis in
$L_{2}[-\omega,\omega]$. Now we take $\omega=2^{j+1}$. 
Since the operator $F_{j}\left(\sqrt{-d^{2}/dx^{2}}\right)$ is self-adjoint one has
$$
\left<F_{j}\left(\sqrt{-d^{2}/dx^{2}}\right), 2^{(-j-1)/2} \delta_{x_{k\pi2^{-j-1}}}\right>=\left<f, F_{j}\left(\sqrt{-d^{2}/dx^{2}}\right)2^{(-j-1)/2} \delta_{x_{k\pi2^{-j-1}}}\right>
$$
and formulas (\ref{norm equality-00}) and (\ref{Sh-equality}) imply that the set of functions 
$$
\psi_{k}^{j}(x)=\left(F_{j}\left(\sqrt{-d^{2}/dx^{2}}\right)2^{(-j-1)/2} \delta_{x_{k\pi2^{-j-1}}}\right)(x)
$$ 
is a Parseval frame in $L_{2}(\mathbb{R})$.  For the same reasons as above we see that $\psi_{k}^{j}\in PW_{2^{j+1}}(\mathbb{R})\cap \>\mathcal{S}(\mathbb{R})$. Moreover, the general frame theory implies that the following reconstruction formula holds
\begin{equation}
f(x)=\sum_{j\in \mathbb{Z}}\sum_{k\in \mathbb{Z}}\left<f, \psi_{k}^{j}\right>\psi^{j}_{k}(x).
\end{equation}

\subsection{Constructing Parseval Paley-Wiener-Schwartz frames using a cubature formula} \label{Thirdmethod}

The following result will be used in this section (compare to a similar statement in \cite{gp}).

\begin{theorem} 
\label{cubformula1}
 If $0<\gamma<1$, $\>\>
\rho<\frac{1}{6}\omega^{-1}\gamma
$ and $\{x_{k}\}$ is such that $\rho/2\leq |x_{k}-x_{k+1}|\leq \rho$
then 
 there exist strictly positive coefficients $\lambda_{x_{k}}>0$, which are of the order $\rho$, \  for which the following equality holds for all functions in $ PW_{\omega}({\mathbb{R}})\cap L_{1}(\mathbb{R})$:
\begin{equation}
\label{cubway_1}
\int_{{\mathbb{R}}}fdx=\sum_{x_{k}}\lambda_{x_{k}}f(x_{k}).
\end{equation}
\end{theorem}

\begin{proof} 
By using Riemann sums and the Fundamental Theorem of calculus we obtain for functions in the Schwartz space the following inequalities
$$
\left|\int_{{\mathbb{R}}}f(x)dx      -        \sum_{x_{k}}f(x_{k})|I_{k}|\right|=
\left|\sum_{k} \int _{I_{k}}f(x)dx-\sum_{x_{k}}\int_{I_{k}}f(x_{k})dx\right| \leq
$$
\begin{equation}
\sum_{k}\int _{I_{k}}\left |f(x)-f(x_{k})\right|\  dx \leq \rho^{1/2}\|f^{'}\|
\end{equation}
Thus, for  $f\in PW_{\omega}(\mathbb{R})$ using the Bernstein inequality and the left side of (\ref{EP-P})
$$
\left|\int_{{\mathbb{R}}}f(x)dx      -        \sum_{x_{k}}f(x_{k})|I_{k}|\right|\leq \rho^{3/2}\|f^{'}\|\leq (1-\gamma)^{-1/2}\rho^{2}\omega\left(\sum_{k}|f(x_{k})|^{2}\right)^{1/2}.
$$
Consider the sampling operator
$$
S: f\rightarrow \{f(x_{k})\},
$$
which maps  $PW_{\omega}(\mathbb{R})$ onto a $V$ which is a subspace of  the space   $\ell^2$ with its standard norm.

If $u \in V$, denote the linear functional $y \to (y,u)$ on $V$ by $\ell_u$.
By our  Plancherel-Polya inequalities (\ref{EP-P}),  the map 
$$
\{f(x_k)\} \to \int _{\mathbb{R}}fdx
$$ 
is a well-defined linear functional on the closed  Hilbert  space
$V$, and so equals $\ell_v$ for some $v \in V$, which may
or may not  have all components positive.  On the other hand, if $w$ is the vector with components $\{|I_{k}|\}$, then $w$ might
not be in $V$, but it has all components positive and of the right size
$$
|I_{k}|   \sim  \rho.
$$
Since, for any vector $u \in V$ the norm of $u$  is exactly the norm of the corresponding functional 
$\ell_u$, the above  inequality 
tells us that 
\begin{equation}
\label{2}
\|Pw-v\| \leq \|w-v\| \leq (1-\gamma)^{-1/2}\rho^{2}\omega,
\end{equation}
where  $P$ is the orthogonal projection onto $V$. Accordingly, if $z$ is the  vector $v-Pw$, then
\begin{equation}
\label{3}
v+(I-P)w = w + z ,
\end{equation}
where $\|z\| \leq  (1-\gamma)^{-1/2}\rho^{2}\omega$.  Note, that all components of the vector $w$ 
 are of order $O(\rho)$, while the order of $\|z\|$ is  $O(\rho^{2})$.  Thus if 
 $\rho\omega$ is sufficiently small, then 
$\lambda := w + z$ has all components positive and of the right size.  Since $\lambda = v + (I-P)w$, the linear
functional $y \to (y,\lambda)$ on $V$ equals $\ell_v$.  In other words, if the vector $\lambda$ has components
 $\{\lambda_{x_{k}}\}, $ then 
$$
\sum_{x_{k}}f(x_{k})\lambda_{x_{k}} = \int_{\mathbb{R}} f dx
$$ 
for all $f \in PW_{\omega}(\mathbb{R})$, as desired.
\end{proof}

The following statement immediately follows from the fact that the Fourier transform of a product of two functions in $L_{2}(\mathbb{R})$ is a convolution of their Fourier transforms. 
\begin{lemma}\label{product}
If $f,g \in PW_{\omega}(\mathbb{R})$ then their product  $fg$ is in $PW_{2\omega}(\mathbb{R})$. In particular,  if $f \in PW_{\omega}(\mathbb{R})$ then $|f|^{2}=f\overline{f}$ belongs to $PW_{2\omega}(\mathbb{R})$. 

\end{lemma}
Thus we can use the above quadrature rule for  $F_{j}\left(\sqrt{-d^{2}/dx^{2}}\right)f=f_{j}$ 
$$
\|f_{j}\|^{2}=\int_{\mathbb{R}}|f_{j}|^{2}dx=\sum_{x_{k}\in M_{\rho}}\lambda_{x_{k}}|f_{j}(x_{k})|^{2}.
$$
Using  (\ref{norm equality-00})  and  introducing functions
$$
\theta_{k}^{j}=\sqrt{\lambda_{x_{k}}}F_{j}\left(\sqrt{-d^{2}/dx^{2}}\right)\delta_{x_{k}}
$$
we obtain a Parseval frame in $L_{2}(\mathbb{R})$ since
$$
\|f||^{2}=\sum_{j}\sum_{k}\left|\left<f_{j}, \theta_{k}^{j}\right>\right|^{2}.
$$
This fact  implies that the following reconstruction formula holds 
$$
f=\sum_{j}\sum_{k}\left<f_{j}, \theta_{k}^{j}\right>\theta_{k}^{j}.
$$

To summarize these examples we list the following crucial facts which were used in the previous constructions:
\begin{enumerate}

\item the Fourier transform provides the Spectral Resolution of the operator $-d/dx$ in $L_{2}(\mathbb{R})$;

\item  existence of the irregular and regular sampling theorems;
\item the fact that the product of two functions in $PW_{\omega}(\mathbb{R})$ is a function in  $PW_{2\omega}(\mathbb{R})$;

\item existence of exact quadrature formula with positive coefficients for Paley-Wiener function.

\end{enumerate}

The goal of our survey is to demonstrate that:
\begin{enumerate}

\item the method developed in section \ref{Firstmethod} can be extended to general Hilbert spaces \cite{Pes01}, to compact Riemannian manifolds \cite{Pes04a}, \cite{Pes04c}, to non-compact manifolds of bounded geometry whose Ricchi curvature is bounded from below \cite{Pes00}, \cite{Pes04b}, to non-compact symmetric Riemannian manifolds \cite{Pes06}, \cite{Pes13b}, to domains in $\mathbb{R}^{n}$ \cite{Pes15c};

\item the method developed in section \ref{Thirdmethod} can be extended to homogeneous compact Riemannian manifolds \cite{gp} and homogeneous manifolds with sub-elliptic structure \cite{Pes15}.

\end{enumerate}

We note that the method of section \ref{Secondmethod} can not be extended to Riemannian manifolds due, in particular, to the lack  of uniformly spaced sets of points.

It should be also mentioned that the methods of \ref{Firstmethod}  were   extended to metric(quantum) and combinatorial graphs \cite{FP}, \cite{Pes05}, \cite{Pes06a}, \cite{Pes08}, \cite{Pes09d}, \cite{Pes10}.

\section{Shannon sampling, Paley-Wiener  frames
and  abstract  Besov subspaces
  }\label{Hilbert}

\subsection{Paley-Wiener vectors in Hilbert spaces}

Consider a self-adjoint positive definite operator
$L$ in a Hilbert space $\Hilb$.
Let $\sqrt{L}$ be the positive square root of $L$.
According to the spectral theory
for such operators \cite{BS}
there exists a direct integral of
Hilbert spaces $X=\int X(\lambda )dm (\lambda )$ and a unitary
operator $\mathcal{F}$ from $\Hilb$ onto $X$, which
transforms the domains of $L^{k/2}, k\in \mathbb{N},$
onto the sets
$X_{k}=\{x \in X|\lambda ^{k}x\in X \}$
with the norm % following norm is finite:
\begin{equation}\label{FT}
\|x(\lambda)\|_{X_{k}}= \left<x(\lambda),x(\lambda)\right>^{1/2}_{X(\lambda)}=\left (\int^{\infty}_{0}
 \lambda^{2k}\|x(\lambda )\|^{2}_{X(\lambda )} dm
 (\lambda ) \right )^{1/2}.
 \end{equation}
% besides
and satisfies the identity
$\mathcal{F}(L^{k/2} f)(\lambda)=
 \lambda ^{k} (\mathcal{F}f)(\lambda), $ if $f$ belongs to the domain of
 $L^{k/2}$.
 We call the operator $\mathcal{F}$ the Spectral Fourier Transform \cite{Pes88b}, \cite{Pes00}. As known, $X$ is the set of all $m $-measurable
  functions $\lambda \mapsto x(\lambda )\in X(\lambda ) $,
  for which the following norm is finite:
$$\|x\|_{X}=
\left(\int ^{\infty }_{0}\|x(\lambda )\|^{2}_{X(\lambda )}dm
(\lambda ) \right)^{1/2} $$
For a function $F$ on $[0, \infty)$ which is bounded and measurable
with respect to $dm$
one can introduce the  operator $\FSL$
by using the formula
\begin{equation}\label{Op-function}
\FSL f=\mathcal{F}^{-1}F( \lambda)\mathcal{F}f,\>\>\>f\in \mathcal{H}.
\end{equation}
If $F$ is real-valued the operator $\FSL$ is self-adjoint.

\begin{remark} \label{rem:sqrtL}
We start with an operator $L$ and switch to $\sqrt{L}$ because
% for the reason that
in many applications (see below) a second-order differential operator $L$ appears first, but it is more natural to work with a first order pseudo-differential operator $\sqrt{L}$. This will become apparent when we discuss sampling density in Subsection \ref{subsect:sampling_density}.

\end{remark}
\begin{definition}
For $\sqrt{L}$  as above  we will  say that a vector $f \in \Hilb$ belongs to the {\it Paley-Wiener space} $\PWoL$
if the support of the Spectral Fourier Transform
$\mathcal{F}f$ is contained in $[0, \omega]$.
\end{definition}
The next two facts are obvious.
\begin{theorem}The spaces $\PWoL$ have the following properties:
\begin{enumerate}
\item  the space $\PWoL$ is a linear closed subspace in
$\Hilb$.
\item the space %set % hgfei
 $\bigcup _{ \omega >0}\PWoL$
 is dense in $\Hilb$;
\end{enumerate}
\end{theorem}
Next we denote by $\mathcal{H}^{k}$  the domain  of $L^{k/2}$.
It is a Banach space,
equipped with the graph norm $\|f\|_{k}=\|f\|+\|L^{k/2}f\|$.
The next theorem contains generalizations of several results
from  classical harmonic analysis (in particular  the
Paley-Wiener theorem). It follows from our  results in
\cite{Pes00} and \cite{Pes09a}.
\begin{theorem}\label{PW proprties}
The following statements hold:
\begin{enumerate}
\item (Bernstein inequality)   $f \in \PWoL$ if and only if
$ f \in \mathcal{H}^{\infty}=\bigcap_{k=1}^{\infty}\mathcal{H}^{k}$,
and the following Bernstein inequalities  holds true
\begin{equation}\label{Bern}
\|L^{s/2}f\|\leq \omega^{s}\|f\| \quad \mbox{for all} \, \,  s\in \mathbb{R}_{+};
\end{equation}
 \item  (Paley-Wiener theorem) $f \in \PWoL$
  if and only if for every $g\in \Hilb$ the scalar-valued function of the real variable  $ t \mapsto
\langle e^{it\sqrt{L}}f,g \rangle $
 is bounded on the real line and has an extension to the complex
plane as an entire function of the exponential type $\omega$;
\item (Riesz-Boas interpolation formula) $f \in \PWoL$ if
and only if  $ f \in \mathcal{H}^{\infty}$ and the
following Riesz-Boas interpolation formula holds for all $\omega > 0$:
\begin{equation}  \label{Rieszn}
i\sqrt{L}f=\frac{\omega}{\pi^{2}}\sum_{k\in\mathbb{Z}}\frac{(-1)^{k-1}}{(k-1/2)^{2}}
e^{i\left(\frac{\pi}{\omega}(k-1/2)\right)\sqrt{L}}f.
\end{equation}
\end{enumerate}
\end{theorem}
\begin{proof}
(1) follows immediately from the definition and representation (\ref{FT}).  To prove (2) it is sufficient to apply the classical Bernstein inequality \cite{Nik} for the  uniform norm on $\mathbb{R}$ to every  function $\langle e^{it\sqrt{L}}f,g\rangle,\>\>g\in \Hilb$. To prove   (3) one has to apply the classical Riesz interpolation formula on $\mathbb{R}$ \cite{Nik} to the functions
$\langle e^{it\sqrt{L}}f,g\rangle$.
\end{proof}
\begin{remark}
For trigonometric polynomials in $L_{p}(\mathbb{T}), \>1\leq p\leq \infty,$ and the operator $\frac{d}{dt}$, the identity (\ref{Rieszn}) was proved by Riesz \cite{Riesz}. For entire functions of exponential type in $L_{p}(\mathbb{R}), \>1\leq p\leq \infty,$ and $\frac{d}{dt}$ this identity was proved by Boas \cite{Boas}.
\end{remark}

\subsection{Frames in Hilbert spaces}
A family of vectors $\{\theta_{v}\}$  in a Hilbert space $\mathcal{H}$ is called a frame if there exist constants
$A, B>0$ such that
\begin{equation}
A\|f\|^{2}\leq \sum_{v}\left|\left<f,\theta_{v}\right>\right|^{2}     \leq B\|f\|^{2} \quad \mbox{for all} \quad f\in \mathcal{H}.
\end{equation}
The largest $A$ and smallest $B$ are called lower and upper frame bounds.

The family of scalars $\{\left<f,\theta_{v}\right>\}$
represents a set of measurements of a vector $f$.
In order to resynthesize the vector $f$
from this collection  of measurements in a linear way
one has to find another
(dual) frame $\{\Theta_{v}\}$.
Then a reconstruction formula is
\begin{equation}
f=\sum_{v}\left<f,\theta_{v}\right>\Theta_{v}.
\end{equation}

Dual frames are not unique in general.
Moreover it may be difficult to find a dual frame in concrete
situations.  %hgfei

If in particular $A=B=1$ the frame is said to be  tight   or Parseval.
Parseval frames are similar in many respects to orthonormal wavelet bases.  For example, if in addition all vectors $\theta_{v}$ are unit vectors, then the frame is an  orthonormal basis.

The main feature of Parseval frames is that
decomposing
and synthesizing a vector from known data are tasks carried out with
the same family of functions, i.e., the Parseval frame is its own dual frame.
The important differences between frames and, say, orthonormal bases is their  redundancy that helps reduce for example the effect of noise in data.

Frames in Hilbert spaces of functions  whose members have simultaneous localization in space and frequency  arise naturally in wavelet analysis on Euclidean spaces,  when the continuous wavelet transforms are discretized.

\subsection{Sampling in abstract Paley-Wiener spaces}

We assume that  there exist a  $C>0$ and  $m_{0}\geq 0$ such that for any $0<\rho<1$ there exists a        set of functionals  $\mathcal{A}^{(\rho)}=\left\{\mathcal{A}_{ k}^{(\rho)}\right\},$  defined on $\mathcal{H}^{m_{0}}$, for which
\begin{equation}\label{A}
\|f\|^{2}\leq
C\left(\sum_{k}\left|\mathcal{A}_{ k}^{(\rho)}(f)\right|^{2}+\rho^{2m}\|L^{m/2}f\|^{2}\right),\>\>\>f\in \mathcal{H}^{m},\>\>m> m_{0},
\end{equation}
       
and

\begin{equation}\label{B1}
\sum_{k}\left|\mathcal{A}_{k}^{(\rho)}(f)\right|^{2}  \leq C \|f\|^{2}\>\> ,\quad \mbox{for all} \, \,  f \in \mathcal{H}^{m}, \>\>m> m_{0}.
\end{equation}

\begin{remark} In notations of section \ref{Firstmethod} if
$$
\mathcal{A}_{k}^{(\rho)}(f)=\sqrt{|I_{k}|}|f(x_{k})|,\>\>\>|I_{k}|\leq \rho,
$$
then inequality (\ref{A})  is similar to (\ref{interm-ineq}) 
and (\ref{B1}) is similar to (\ref{left-side}).

\end{remark}
\begin{remark}
Following \cite{Pes01}, \cite{Pes04b} we call inequality (\ref{A}) a  {\it Poincar\'{e}-type inequality} since it is an estimate of the norm  of % a vector
$f$ through the norm  of its ``derivative'' $L^{m/2}f$.
\end{remark}

 Let us introduce vectors  $\mu_{ k}\in \Hilb$ such that
 $$
 \left<f,\mu_{ k}\right>=\mathcal{A}_{k}^{(\rho)}(f),\>\>f\in \mathcal{H}^{m},\>\>m> m_{0}.
 $$
 Let  $\mathcal{P}_{\negthinspace \omega}$
 be the orthogonal projection of $\Hilb$ onto
 $\PWoL$ and put
 \begin{equation}\label{functionals-phi}
 \phi^{\omega}_{ k}=\mathcal{P}_{\negthinspace \omega}\mu_{ k}.
\end{equation}
 Using  the Bernstein inequality (\ref{Bern}) we obtain  the following statement.

 \begin{theorem}\label{Ss}(Sampling Theorem)

 Assume  that assumptions (\ref{A}) and (\ref{B1}) are
 satisfied and for a given  $\omega>0$ and $\delta \in (0,1)$  pick a $\rho$ such that 
 
$$
\rho^{2m}=C^{-1}\omega^{-2m}\delta.
$$

 Then  the family of vectors
$\{\phi^{\omega}_{k}\}$ in (\ref{functionals-phi}) is a frame for the Hilbert space $\PWoL$ and
\begin{equation}\label{frame-in-PW}
(1-\delta) \|f\|^{2}\leq \sum_{k}\left|\left<f,\phi^{\omega}_{k}\right>\right|^{2}     \leq \|f\|^{2},\>\>\>f\in \PWoL.
\end{equation}
The canonical dual frame $\{\Theta^{\omega}_{k}\}$  has the property   $\Theta^{\omega}_{k}\in  \PWoL$ and
provides the following reconstruction formulas
\begin{equation}
f=\sum_{k}\left<f,\phi^{\omega}_{k}\right>\Theta^{\omega}_{k}=\sum_{k}\left<f,\Theta^{\omega}_{k}\right>\phi^{\omega}_{k},\>\>\>f\in \PWoL.
\end{equation}
\end{theorem}

\subsection{Partitions of unity on the frequency side}

\label{subsect:part_unity_freq}

The construction of frequency-localized frames is typically achieved via spectral calculus. The idea is to start from a partition of unity on the positive real axis. In the following, we will be considering two different types of such partitions, whose construction we now describe in some detail.

Now we are going to construct partitions of unity $F_{j}$ and $G_{j}=F^{2}_{j}$ which are similar to one that were introduced in (\ref{functions2})-(\ref{partsums4}).

 Let $g\in C^{\infty}(\mathbb{R}_{+})$ be a non-increasing
%  monotonic
 function such that $supp(g)\subset [0,\>  2], $ and $g(\lambda)=1$ for $\lambda\in [0,\>1], \>0\leq g(\lambda)\leq 1, \>\lambda>0.$
 We now let
\[
 h(\lambda) = g(\lambda) - g(2 \lambda)~,
\] which entails $supp(h) \subset [2^{-1},2]$,
and use this to define
\begin{equation} \label{functions}
 F_0(\lambda) = \sqrt{g(\lambda)}~, F_j(\lambda) = \sqrt{h(2^{-j} \lambda)}~, j \ge 1~,
\end{equation}
as well as
\[
 G_j(\lambda) = \left[F_j(\lambda)\right]^2=F_j^2(\lambda)~, j \ge 0~.
\]
As a result of the definitions, we get for all $\lambda \ge 0$ the equations
\begin{equation} \label{partsums1}
\sum_{j = 0}^n G_j(\lambda) = \sum_{j = 0}^n F_j^2(\lambda)
=  g(2^{-n} \lambda),
\end{equation}
and as a consequence
\begin{equation} \label{partsums2}
\sum_{j \ge 0} G_j(\lambda) = \sum_{j \ge 0} F_j^2(\lambda) =  1~,\>\>\>\lambda\geq 0,
\end{equation}
 with finitely many nonzero terms occurring in the sums for each
 fixed $\lambda$. %uuu
We call the sequence $(G_j)_{j \ge 0}$ a {\bf (dyadic) partition of unity}, and $(F_j)_{j \ge 0}$ a {\bf quadratic (dyadic) partition of unity}.  As will become soon apparent, quadratic partitions are useful for the construction of frames.

 Using the spectral theorem one has
 % {\sc here SLB}
$$
F_{j}^{2}\SLB  f=\mathcal{F}^{-1}\left(F_{j}^{2}(\lambda)\mathcal{F}f(\lambda)\right),\>\>\>j\geq 1,
$$
and thus
\begin{equation} \label{eqn:quad_part_identity}
 f = \mathcal{F}^{-1}\mathcal{F}f(\lambda) =\mathcal{F}^{-1}\left(\sum_{j\geq 0}F_{j}^{2}(\lambda)\mathcal{F}f(\lambda)\right) = \sum_{j\geq 0} F_{j}^2\SLB f
\end{equation}  % {\sc here SLB}
Taking inner product with $f$ gives
$$
\|F_{j}\SLB f\|^{2}=\langle F_{j}^{2}\SLB f, f \rangle
$$  % {\sc here SLB}
and
\begin{equation}
\label{norm equality-0}
\|f\|^2=\sum_{j\geq 0}\langle F_{j}^2 \SLB f,f\rangle=\sum_{j\geq 0}\|F_{j}\SLB f\|^2 .
\end{equation}
Similarly, we get the identity
\[
 \sum_{j \geq 0} G_j \SLB f = f~.
\]
Moreover, since the functions $G_j,  F_{j}$, have their supports in  $
[2^{j-1},\>\>2^{j+1}]$, the elements $ F_{j} \SLB f $ and $G_j \SLB f$
 are bandlimited to  $[2^{j-1},\>\>2^{j+1}]$, whenever $j \ge 1$, and to $[0,2]$ for $j=0$.

 \subsection{Paley-Wiener frames in  Hilbert spaces}\label{Hilb}

Using the notation from above and Theorem \ref{Ss}, one can describe  the following Paley-Wiener frame in an abstract Hilbert space $\mathcal{H}$.

\begin{theorem}\label{frameH}(Paley-Wiener nearly Parseval frame in $\Hilb$)

For a fixed $\delta\in (0,1)$ and $j\in \mathbb{N}$ let $\{\phi^{j}_{k}\}$ be a set of vectors described in Theorem \ref{Ss} that correspond to $\omega=2^{j+1}$. 
Then for functions $F_{j}$   introduced in (\ref{functions})   the family of Paley-Wiener  vectors
$$
\Phi^{j}_{k}= F_{j} \SLB \phi^{j}_{k}
$$
has the following properties:
\begin{enumerate}
\item Each vector $\Phi^{j}_{k}$ belongs  to  $\bPW_{[2^{j-1},\>2^{j+1}]}(\sqrt{L}) ,\>\> j \in   N, \>k=1,...,$
\item   The family $\left\{\Phi^{j}_{k}\right\}$ is  a frame in $\Hilb$ with constants $1-\delta$ and $1$:
\begin{equation}
(1-\delta)\|f\|^2\leq \sum_{j\geq 0}\sum_{k}\left|\left< f, \Phi^{j}_{k}\right>\right|^{2}\leq \|f\|^2,\>\>\>f \in \Hilb.
\end{equation}
\item  The canonical dual frame $\{\Psi^{j}_{k}\}$
also consists of bandlimited  vectors $\Psi^{j}_{k}\in \bPW_{[2^{j-1},\>2^{j+1}]}\SLB ,\>\>j\in  [0,\>\infty), \>k=1,...,$ and satisfies the inequalities
\begin{equation}
\|f\|^2\leq \sum_{j\geq 0}\sum_{k}\left|\left< f, \Psi^{j}_{k}\right>\right|^{2}\leq (1-\delta)^{-1} \|f\|^2,\>\>\>f\in \Hilb.
\end{equation}

\item The reconstruction formulas hold for every $f\in \mathcal{H}$

\begin{equation}
f=\sum_{j}\sum_{k}\left<f,\Phi^{j}_{k}\right>\Psi^{j}_{k}=\sum_{j}\sum_{k}\left<f,\Psi^{j}_{k}\right>\Phi^{j}_{k}.
\end{equation}

\end{enumerate}
\end{theorem}

The last item here follows from the general theory of frames \cite{Gr}.
We also note that  for reconstruction of a Paley-Wiener vector  from  a set of samples one can use, besides dual frames, the  variational (polyharmonic) splines in Hilbert spaces developed in   \cite{Pes98a}-\cite{Pes09b}.

\subsection{Interpolation and  Approximation spaces}\label{Interp}
\noindent
The goal of the section is to establish certain  connections
between interpolation spaces and approximation spaces to
be used later.  These connections are well known to specialists and   due to  Peetre-Sparr \cite{PS} and Butzer-Scherer \cite{BSch}.
However,  we formulate and prove these relations in a form which is most suitable for our purposes (see  \cite{KP}). In particular, our formulation is more general than a similar Theorem 9.1 in Ch. 7 in \cite{DVL}.  The result of primary interest for the following is Theorem \ref{equivalence-interpolation}, and readers who are not interested in its proof can safely skip this section. 

The general theory of interpolation spaces can be found
in \cite{BL}, \cite{BB},  \cite{KPS}. The notion of
approximation spaces and their relations to interpolations spaces
is described in \cite{BL},   \cite{BS},  \cite{DVL}, \cite{PS}.

It is important to realize that the relations between
interpolation and approximation spaces cannot be described
in  the language of normed spaces. We have to make use of
quasi-normed linear spaces in order to treat them
simultaneously.

A quasi-norm $\|\cdot\|_{\EB}$ on linear space $\EB$ is
a real-valued function on $\EB$ such that for any
$f,f_{1}, f_{2} \in \EB$ the following holds true
\begin{enumerate}
\item $\|f\|_{\EB}\geq 0;$
\item $\|f\|_{\EB}=0  \Longleftrightarrow   f=0;$
\item $\|-f\|_{\EB}=\|f\|_{\EB};$
\item there exists some $C_{\EB} \geq 1$ such that
$\|f_{1}+f_{2}\|_{\EB}\leq C_{\EB}(\|f_{1}\|_{\EB}+\|f_{2}\|_{\EB}).$
\end{enumerate}

%{\sc moved up!}

For a general quasi-normed linear spaces  $\EB$ the notation
$(\EB)^{\rho},\>\rho>0,$ is used for the space $\EB$ endowed with the
quasi-norm $\|\cdot\|^{\rho}$.

Two quasi-normed linear spaces $\EB$ and $\FB$ form a
pair if they are linear subspaces of a common linear space
$\AB$ and the conditions
$\|f_{k}-g\|_{\EB}\rightarrow 0,$ and
$\|f_{k}-h\|_{\FB}\rightarrow 0$
imply equality $g=h$ (in $\AB$).
For any such pair $\EB,\FB$ one can construct the
space $\EB \cap \FB$ with quasi-norm
$$
\|f\|_{\EB \cap \FB}=\max\left(\|f\|_{\EB},\|f\|_{\FB}\right)
$$
and the sum of the spaces,  $\EB + \FB$ consisting of all sums $f_0+f_1$ with $f_0 \in \EB, f_1 \in \FB$, and endowed with the quasi-norm
$$
\|f\|_{\EB + \FB}=\inf_{f=f_{0}+f_{1},f_{0}\in \EB, f_{1}\in
\FB}\left(\|f_{0}\|_{\EB}+\|f_{1}\|_{\FB}\right).
$$

Quasi-normed spaces $\HB$ with
$\EB \cap \FB \subset \HB \subset \EB + \FB$
are called intermediate between $\EB$ and $\FB$.
If both $E$ and $F$ are complete the
inclusion mappings are automatically continuous. %HGFei
An additive homomorphism $T: \EB \rightarrow \FB$
is called bounded if
$$
\|T\|=\sup_{f\in \EB,f\neq 0}\|Tf\|_{\FB}/\|f\|_{\EB}<\infty.
$$
An intermediate quasi-normed linear space $\HB$
interpolates between $\EB$ and $\FB$ if every bounded homomorphism $T:
\EB+\FB \rightarrow \EB + \FB$
which is a bounded homomorphism of $\EB$ into
$\EB$ and a bounded homomorphism of $\FB$ into $\FB$
is also a bounded homomorphism of $\HB$ into $\HB$

On $\EB+\FB$ one considers the so-called Peetre's $K$-functional
\begin{equation}
K(f, t)=K(f, t,\EB, \FB)=\inf_{f=f_{0}+f_{1},f_{0}\in \EB,
f_{1}\in \FB}\left(\|f_{0}\|_{\EB}+t\|f_{1}\|_{\FB}\right).\label{K}
\end{equation}
The quasi-normed linear space $(\EB,\FB)^{K}_{\theta,q}$,
with parameters $0<\theta<1, \,
0<q\leq \infty$,  or $0\leq\theta\leq 1, \, q= \infty$,
is introduced as the set of elements $f$ in $\EB+\FB$ for which
\begin{equation}
\|f\|_{\theta,q}=\left(\int_{0}^{\infty}
\left(t^{-\theta}K(f,t)\right)^{q}\frac{dt}{t}\right)^{1/q} < \infty .\label{Knorm}
\end{equation}

It turns out that $(\EB,\FB)^{K}_{\theta,q}$
% , 0<\theta<1, 0\leq q\leq
% \infty,$ or $0\leq\theta\leq 1,  q= \infty,$
% redundant!! hgfei
with the quasi-norm
(\ref{Knorm})  interpolates between $\EB$ and $\FB$. The following
Reiteration Theorem is one of the main results of the theory (see
\cite{BL}, \cite{BB}, \cite{KPS}, \cite{PS}).
\begin{theorem}
Suppose that $\EB_{0}, \EB_{1}$ are complete intermediate quasi-normed
linear spaces for the pair $\EB,\FB$.  If $\EB_{i}\in
\mathcal{K}(\theta_{i}, \EB, \FB)$ which means
$$
K(f,t,\EB,\FB)\leq Ct^{\theta_{i}}\|f\|_{\EB_{i}}, i=0,1,
$$
where $ 0\leq \theta_{i}\leq 1, \theta_{0}\neq\theta_{1},$ then
$$
(\EB_{0},\EB_{1})^{K}_{h,q}\subset (\EB,\FB)^{K}_{\theta,q},
$$
where $0<q<\infty, 0<h<1,
\theta=(1-h)\theta_{0}+h\theta_{1}$.

 If for the same pairs
$\EB,\FB$ and  $\EB_{0}, \EB_{1}$  one has $\EB_{i}\in
\mathcal{J}(\theta_{i}, \EB, \FB)$, which means
$$
\|f\|_{\EB_{i}}\leq C\|f\|_{\EB}^{1-\theta_{i}}\|f\|_{\FB}^{\theta_{i}},
i=0,1,
$$
where $ 0\leq \theta_{i}\leq 1, \theta_{0}\neq\theta_{1},$
then for any parameter $0<q<\infty, 0<h<1$
$$
(\EB,\FB)^{K}_{\theta,q} \subset (\EB_{0},\EB_{1})^{K}_{h,q} \quad
\mbox{for} \, \, \theta=(1-h)\theta_{0}+h\theta_{1}.
$$
% where % $0<q<\infty, 0<h<1,
\end{theorem}

It is important to note that in all cases which will be considered
in the present article the space $\FB$ will be continuously embedded
as a subspace into $\EB$. In this case (\ref{K}) can be introduced
by the formula
$$
K(f, t)=\inf_{f_{1}\in \FB}
\left(\|f-f_{1}\|_{\EB}+t\|f_{1}\|_{\FB}\right),
$$
which implies the inequality
\begin{equation}
K(f, t)\leq \|f\|_{\EB}.
\end{equation}
This inequality can be used to show  that the norm (\ref{Knorm})
is equivalent to the norm
\begin{equation}
\|f\|_{\theta,q}=\|f\|_{\EB}+\left(\int_{0}^{\varepsilon}
\left(t^{-\theta}K(f, t)\right)^{q}\frac{dt}{t}\right)^{1/q}
% , \varepsilon>0,
\end{equation}
for any positive $\varepsilon$.

Let us introduce another functional on $\EB+\FB$,
where $\EB$ and $\FB$ form a pair of quasi-normed linear spaces
$$
\mathcal{E}(f, t)=\mathcal{E}(f, t,E, F)=\inf_{g\in \FB,
\|g\|_{\FB}\leq t}\|f-g\|_{\EB}.
$$
\begin{definition}
The approximation space $\mathcal{E}_{\alpha,q}(\EB, \FB),
0<\alpha<\infty, 0<q\leq \infty $ is the quasi-normed linear spaces
of all $f\in \EB+\FB$ for which the quasi-norm
\begin{equation}
\| f \|_{\mathcal{E}_{\alpha,q}(\EB, \FB)} = \left(\int_{0}^{\infty}\left(t^{\alpha}\mathcal{E}(f,
t)\right)^{q}\frac{dt}{t}\right)^{1/q} 
\end{equation}
is finite. 
\end{definition}

The following theorem describes the relations between interpolation
and approximation spaces (see  \cite{BL}, Ch. 7).
\begin{theorem}
For $\theta=1/(\alpha+1)$ and $r=\theta q$   one has
$$
(\mathcal{E}_{\alpha, r}(\EB, \FB))^{\theta}=
(\EB,\FB)^{K}_{\theta,q}.
$$
\end{theorem}
The following important result is known as the Power Theorem (see
\cite{BL}, Ch. 3).
\begin{theorem}\label{Powerthm}
%
% ALTERNATIVE FORMULATION
%
Given positive values $ \rho_{0}>0, \rho_{1}>0$
and interpolation parameters $0 < \theta < 1$
and $0 < q \leq \infty$, the interpolation
space obtained from  powers of two spaces is the same as the
power of an interpolation space: For the
parameter values  $\rho = (1-\theta)\rho_0 + \theta \rho_1,
q' = \rho q$
and  $\theta' =  \theta \rho_1/\rho$, we have
$$
\left((\EB)^{\rho_{0}}, (\FB)^{\rho_{1}}\right)^{K}_{\theta,q}
= \left((\EB,\FB)^{K}_{\theta', q'}\right)^{\rho}.
$$
\end{theorem}
The next theorem   represents a very abstract version of what is  known as an Equivalence Approximation Theorem.

\begin{theorem}\label{equivalence-interpolation}
 Suppose that $\mathcal{T}\subset \FB \subset \EB$ are quasi-normed
linear spaces and $\EB$ and $\FB$ are complete.
If there exist $C>0$ and $\beta >0$ such that
the following Jackson-type inequality is satisfied
\begin{equation}
t^{\beta}\mathcal{E}(t,f,\mathcal{T},\EB)\leq C\|f\|_{\FB},
\label{dir}
 \, \, t>0, \, \, \mbox{for all} \, \, f \in \FB,
\end{equation}
 then the following embedding holds true
\begin{equation}\label{imbd-1}
(\EB,\FB)^{K}_{\theta,q}\subset
\mathcal{E}_{\theta\beta,q}(\EB, \mathcal{T}), \quad \>0<\theta<1, \>0<q\leq \infty.
\end{equation}
If there exist $C>0$ and $\beta>0$
such that % for any $f\in \mathcal{T}$ the following
the following Bernstein-type inequality holds
\begin{equation}\label{bern}
\|f\|_{\FB}\leq C\|f\|^{\beta}_{\mathcal{T}}\|f\|_{\EB}
\quad \mbox{for all}  \, \, f\in \mathcal{T},
\end{equation}
then the following embedding holds true
\begin{equation}\label{imbd-2}
\mathcal{E}_{\theta\beta, q}(\EB, \mathcal{T})\subset
(\EB, \FB)^{K}_{\theta, q}  , \quad 0<\theta<1, \>0<q\leq \infty.
\end{equation}
\end{theorem}\label{intthm}

\begin{proof}
According to  (\cite{BL}, Ch.7) one has for any $s>0$
\begin{equation}
t=K_{\infty}(f,s)=
K_{\infty}(f,s,\mathcal{T},\EB)=\inf_{f=f_{1}+f_{2}, f_{1}\in
\mathcal{T}, f_{2}\in \EB}\max(\|f_{1}\|_{\mathcal{T}},
s\|f_{2}\|_{\EB})\label{t}
\end{equation}
the following inequality holds
\begin{equation}
s^{-1}K_{\infty}(f, s)\leq \lim_{\tau\rightarrow
t-0}\inf\mathcal{E}(f, \tau,\EB,\mathcal{T})\label{lim}.
\end{equation}
Since
\begin{equation}
K_{\infty}(f,s)\leq K(f, s)\leq 2K_{\infty}(f, s),\label{equiv}
\end{equation}
the Jackson-type inequality (\ref{dir}) and the inequality
 (\ref{lim}) imply
\begin{equation}
s^{-1}K(f,s,\mathcal{T},\EB)\leq Ct^{-\beta}\|f\|_{\FB}.
\end{equation}
The equalities (\ref{t}) and (\ref{equiv}) imply the
estimate
\begin{equation}
t^{-\beta}\leq 2^{\beta}
\left(K(f,s,\mathcal{T},\EB)\right)^{-\beta}
\end{equation}
which along with the previous inequality gives the estimate
$$
K^{1+\beta}(f,s,\mathcal{T},\EB)\leq C s \|f\|_{\FB}
$$
which in turn implies the inequality
\begin{equation}
K(f,s,\mathcal{T},\EB)\leq C s^{\frac{1}{1+\beta}}
\|f\|_{\FB}^{\frac{1}{1+\beta}}.\label{class1}
\end{equation}
At the same time one has
\begin{equation}
K(f, s,\mathcal{T},\EB)= \inf_{f=f_{0}+f_{1},f_{0}\in \mathcal{T},
f_{1}\in \EB}\left(\|f_{0}\|_{\mathcal{T}}+s\|f_{1}\|_{\EB}\right)\leq
s\|f\|_{\EB},\label{class2}
\end{equation}
for every $f \in \EB$. Inequality (\ref{class1}) means that the
quasi-normed linear space $(\FB)^{\frac{1}{1+\beta}}$ belongs to the
class $\mathcal{K}(\frac{1}{1+\beta}, \mathcal{T}, \EB)$ and
(\ref{class2}) means that the quasi-normed linear space $\EB$
belongs to the class $\mathcal{K}(1, \mathcal{T}, \EB)$. This fact
allows to use the Reiteration Theorem to obtain the embedding
\begin{equation}
\left((\FB)^{\frac{1}{1+\beta}},\EB\right)^{K}_{\frac{1-\theta}{1+\theta
\beta},q(1+\theta \beta)}\subset \left(\mathcal{T},
\EB \right)^{K}_{\frac{1}{1+\theta \beta},q(1+\theta \beta)}
\end{equation}
for every $0<\theta<1, 1<q<\infty$. But the space on the left is
the space
$$
\left(\EB,(\FB)^{\frac{1}{1+\beta}}\right)^{K}_{\frac{\theta(1+\beta)}{1+\theta
\beta},q(1+\theta \beta)},
$$
which, according to the Power Theorem, is the space
$$
\left((\EB,\FB)^{K}_{\theta, q}\right)^{\frac{1}{1+\theta \beta}}.
$$
All these results along with the equivalence of  interpolation and
approximation spaces give  the embedding
$$
\left(\EB,\FB\right)^{K}_{\theta,q}\subset
\left(\left(\mathcal{T},\EB\right)^{K}_{\frac{1}{1+\theta \beta},q(1+\theta \beta)}\right)^{1+\theta \beta}
=\mathcal{E}_{\theta \beta,q}(\EB,\mathcal{T}),
$$
which proves  the embedding (\ref{imbd-1}).

Conversely,  the Bernstein-type  inequality (\ref{bern}) implies for   $\gamma =\frac{1}{1+\beta}$  the inequality
\begin{equation}\label{bern-1}
\|f\|_{\FB}^\gamma \leq
C\|f\|^\gamma_{\mathcal{T}}\|f\|_{\EB}^\gamma.
\end{equation}
Along with the obvious equality $\|f\|_{\EB}=\|f\|^{0}_{\mathcal{T}}\|f\|_{\EB}$
and the Iteration Theorem one obtains the embedding
$$ \left(\mathcal{T},
\EB \right)^{K}_{\frac{1}{1+\theta \beta},q(1+\theta
\beta)}\subset \left((\FB)^{\frac{1}{1+\beta}},\EB\right)^{K}_{\frac{1-\theta}{1+\theta
\beta},q(1+\theta \beta)}.
$$
In order to derive the embedding (\ref{imbd-2}), one can then use the
same arguments as above. This completes the proof.
\end{proof}

\subsection{Besov subspaces in  Hilbert spaces}\label{AbstractBesov}

We introduce the inhomogeneous  Besov space
$ \mathcal{B}_{\mathcal{H},q}^\alpha \SLB$
as an interpolation space between the Hilbert space
$\mathcal{H}$ and "Sobolev  space" $H^{r}$, defined as the 
domain of the operator $(I+L)^{r/2}$ endowed with the graph norm, 
where
 $r$ can be any natural number such that $0<\alpha<r, 1\leq
q\leq\infty$. More precisely, we have \cite{BB}, \cite{T3},
$$
\mathcal{B}_{\mathcal{H},q}^\alpha \SLB =( \mathcal{H},H^{r})^{K}_{\theta, q},\>\>\> 0<\theta=\alpha/r<1,\>\>\>
1\leq q\leq \infty.
$$
% where $K$ is the Peetre's interpolation functor.
% redundant comment, hgfei

We also introduce a notion of best approximation
\begin{equation}
\mathcal{E}(f,\omega)=\inf_{g\in
\PWoL}\|f-g\|_{\mathcal{H}}.\label{BA1}
\end{equation}

Our goal is to apply Theorem \ref{equivalence-interpolation} in the situation where $E=\mathcal{H}$,  $\>F=H^{r}$ and $\mathcal{T}=\PWoL $ is a natural abelian group as the additive group of a vector space, with the quasi-norm
\[
 \| f \|_{\mathcal{T}} = \inf \left \{ \omega'>0~: f \in \mathbf{PW}_{\mathbf{\omega}'}\left(\sqrt{L}\right) \right\}~.
\]

To be more precise it is the space
of finite sequences of Fourier coefficients $\mathbf{c}=(c_{1},...c_{m})\in \PWoL$
 where $m$ is the greatest index such that the eigenvalue $\lambda_{m}\leq \omega$.
 The quasi-norm $\|\mathbf{c}\|_{\PWoL} $ where $\mathbf{c}=(c_{1},...c_{m})\in \PWoL$  is defined as square root from the highest eigenvalue  $\lambda_{j}$ for which the corresponding Fourier coefficient $c_{j}\neq 0$ but $c_{j+1}=...=c_{m}=0$:
 $$
 \|\mathbf{c}\|_{E_{\omega}(L)} =\|(c_{1},...c_{m})\|_{E_{\omega}(L)}=\max\left\{\sqrt{\lambda_{j}}: c_{j}\neq 0, \>\>c_{j+1}=...=c_{m}=0\right\}.
 $$

\begin{remark}
Let us emphasize  that the reason we need the language of  quasi-normed spaces is because $\| \cdot \|_{\mathcal{T}} $ is clearly not a norm, only a quasi-norm on $ \PWoL$. 
\end{remark}

The Plancherel Theorem allows us to verify a  generalization of the Bernstein inequality for bandlimited functions in $ f \in \PWoL$.
\begin{lemma} (\cite{Pes88b}, \cite{Pes00})
A vector $f$ belongs to the space $ \PWoL$ if and only if the following Bernstein inequality holds
$$
\|L^{r/2}f\|_{\mathcal{H}}\leq \omega ^{r}\|f\|_{\mathcal{H}},\>\>\>r\in \mathbb{R}_{+}.
$$

\end{lemma}
\begin{proof}

Assume that  vector $f$ belongs to the space $ \PWoL$ and $\mathcal{F}f=x\in X$.  Then

$$\left(\int ^{\infty}_{0}\lambda
^{2r}\|x(\lambda)\|^{2}_{X(\lambda)}dm(\lambda
)\right)^{1/2}=
\left(\int^{\omega}_{0}\lambda^{2r}\|x(\lambda
)\|^{2}_{X(\lambda)}dm(\lambda)\right)^{1/2} \leq \omega
^{r}\|x\|_{X} , r\in \mathbb{R}_{+} , 
$$ 
which gives Bernstein inequality for $f$.

Conversely, if $f$ satisfies Bernstein inequality then $x=\mathcal{F}f$ satisfies
$\|x\|_{X_{k}} \leq
\omega^{k}\|x\|_{X}.$ Suppose that there exists a set $\sigma
\subset [ 0, \infty ]\setminus [ 0,
\omega]$ whose $m $-measure is not zero and $x|_{\sigma }\neq 0.$ We can
assume
that $\sigma \subset
[\omega +\epsilon , \infty )$ for some $\epsilon >0.$ Then for any $r\in
\mathbb{R}_{+}$ we have

$$ \int _{\sigma }\|x(\lambda )\|^{2}_{X(\lambda)}dm (\lambda ) \leq \int
^{\infty }_{\omega +\epsilon}\lambda ^{-2r}\|
\lambda ^{r}x(\lambda)\|^{2}_{ X(\lambda)}d\mu
\leq\|x\|^{2}_{X}\left(\omega /\omega
+\epsilon \right)^{2r}$$ which shows that or $x(\lambda)$ is zero on
$\sigma $ or $\sigma $ has measure
zero. 
\end{proof}

One also has an analogue of the Jackson inequality
$$
\mathcal{E}(f,\omega)\leq \omega^{-r}\|f\|_{H^{r}},\>\>\>f\in H^{r}.
$$
 Indeed, as in \cite{Pes88b} one has for $f\in H^{r}$
$$
\mathcal{E}(f,\omega)\leq\left( \int_{\omega}^{\infty}\|x(\lambda)\|^{2}_{X(\lambda)}dm(\lambda)\right)^{1/2}=
$$
$$
\left( \int_{\omega}^{\infty}\lambda^{-2r}\lambda^{2r}\|x(\lambda)\|^{2}_{X(\lambda)}dm(\lambda)\right)^{-r}\leq 
\omega^{-r}\|L^{r/2}f\|_{\mathcal{H}},\>\>\>r\in \mathbb{R}_{+}.
$$

These two inequalities and Theorem \ref{equivalence-interpolation}  imply the following result (compare to  \cite{Pes88a}, \cite{Pes09b}, \cite{Pes11}).

\begin{theorem} \label{approx}
For $\alpha>0, 1\leq q\leq\infty$ the norm of
 $\mathcal{B}_{\mathcal{H},q}^{\alpha}\SLB$,
  is equivalent to
\begin{equation}
\|f\|_{\mathcal{H}}+\left(\sum_{j=0}^{\infty}\left(2^{j\alpha }\mathcal{E}(f,
2^{j})\right)^{q}\right)^{1/q}.
\end{equation}
\label{maintheorem1}
\end{theorem}
Let the functions $F_{j}$ be as in Subsection \ref{subsect:part_unity_freq}. Note, that
\begin{equation}\label{range}
F_j\SLB: \mathcal{H} \rightarrow \bPW_{[2^{j-1},  2^{j+1}]}\SLB,\>\>\> \|F_j\SLB\ \|\leq1,
 \end{equation}
and
$$
\sum_{j\geq 0}F_{j}^{2}(\lambda)=1, \>\>\lambda\geq 0.
$$

\begin{theorem}\label{projections}
For $\alpha>0, 1\leq q\leq\infty$ the norm of
 $\mathcal{B}_{\mathcal{H},q}^{\alpha}\SLB$,
  is equivalent to

\begin{equation}
f \mapsto \left(\sum_{j=0}^{\infty}\left(2^{j\alpha
}\left \|F_j\SLB f\right \|_{\mathcal{H}}\right)^{q}\right)^{1/q},
\label{normequiv-1}
\end{equation}
  with the standard modifications for $q=\infty$.
\end{theorem}

\begin{proof}
Recall from (\ref{eqn:quad_part_identity}) that
$$
\|f\|^2=\sum_{j\geq 0}\left\|F_{j} \SLB f\right\|^2 .
$$
We obviously have
$$
\mathcal{E}(f, 2^{l})\leq \sum_{j> l} \left \|F_j \SLB f\right \|_{\mathcal{H}}.
$$
By using a discrete version of Hardy's inequality \cite{BB} we obtain the estimate
\begin{equation} \label{direct}
\|f\|+\left(\sum_{l=0}^{\infty}\left(2^{l\alpha }\mathcal{E}(f,
2^{l})\right)^{q}\right)^{1/q}\leq C \left(\sum_{j=0}^{\infty}\left(2^{j\alpha
}\left \|F_j \SLB f\right \|_{\mathcal{H}}\right)^{q}\right)^{1/q}.
\end{equation}
Conversely,
 for any $g\in \bPW_{2^{j-1}} \SLB$ we have
$$
\left\|F_j\SLB f\right\|_{\mathcal{H}}=\left\|F_{j}\SLB (f-g)\right\|_{\mathcal{H}}\leq \|f-g\|_{\mathcal{H}}.
$$
% It gives the inequality
This implies the estimate
$$
\left\|F_j\SLB f\right\|_{\mathcal{H}}\leq \mathcal{E}(f,\>2^{j-1}),
$$
which shows that the inequality opposite to (\ref{direct}) holds.
 This completes the proof.
\end{proof}

 \begin{theorem}\label{mainLemma2}
For $\alpha>0, 1\leq q\leq\infty$ the norm of
 $\mathcal{B}_{\mathcal{H},q}^{\alpha}\SLB$
  is equivalent to
\begin{equation}
 \left(\sum_{j=0}^{\infty}2^{j\alpha q }
\left(\sum_{k}\left|\left<f,\Phi^{j}_{k}\right>\right|^{2}\right)^{q/2}\right)^{1/q}\asymp \|f\|_{B_{q}^{\alpha}},
\label{normequiv}
\end{equation}
  with the standard modifications for $q=\infty$.
\end{theorem}
\begin{proof}
For   $f\in \mathcal{H}$  and operator $F_{j}\SLB$ 
we apply (\ref{frame-in-PW}) to $F_{j}\SLB f\in \bPW_{2^{j+1}}\SLB$ to obtain
\begin{equation}
(1-\delta)\left \|F_j\SLB f\right \|_{\mathcal{H}}^{2}\leq
\sum_k\left|\left< F_{j}\SLB f, \phi^{j}_{k}\right>\right|^{2}\leq
\left \|F_j\SLB f\right \|_{\mathcal{H}}^{2}.
\end{equation}
 Since $\Phi^{j}_{k}=F_{j}\SLB \phi^{j}_{k}$
 we obtain the following inequality
 % for any $f\in \mathcal{H}$ %uuu
$$
\sum_{k}\left|\left<f,\Phi^{j}_{k}\right>\right|^{2}\leq \left \|F_j\SLB f\right \|_{\mathcal{H}}^{2}\leq \frac{1}{1-\delta}
\sum_{k}\left|\left<f,\Phi^{j}_{k}\right>\right|^{2}
\quad \mbox{for all} \, \, \, f\in \mathcal{H}.
$$
Our statement follows now from Theorem \ref{projections}.
\end{proof}

\newpage

\section{Manifolds, function spaces and operators}\label{Manifolds}

The goal of this section is to introduce some "real life" situations in which we develop space-frequency analysis.

\subsection{Riemannian manifolds without boundary}\label{sub-1}

Let $\mfd{M}$, dim$\>\mfd{M}=n$, be a connected $C^{\infty}-$smooth Riemannian
manifold with a $(2,0)$ metric tensor $g$ that defines an inner
product on every tangent space $T_{x}(\mfd{M}), x\in \mfd{M}$. The
corresponding Riemannian distance $d$ on $\mfd{M}$ and
the Riemannian measure $d\mu$ on $\mfd{M}$ are given by
\begin{equation} \label{Riemmeas1}
d(x,y)=\inf \int_{a}^{b}
\sqrt{g\left(\frac{d\alpha}{dt},\frac{d\alpha}{dt}\right)}dt,
\>\>\>\>d\mu=\sqrt{det(g_{ij})}dx,
\end{equation}
where the infimum is taken over all $C^{1}-$curves
$\alpha:[a,b]\rightarrow \mfd{M}, \alpha(a)=x, \alpha(b)=y,$
the $\{g_{ij}\}$ are the components of the tensor $g$ in a
local coordinate system and  $dx$ is the Lebesgue measure on
$\mathbb{R}^{d}$.  Let $ exp_ {x} $ : $T_{x}(\mfd{M})\rightarrow \mfd{M},$ be the exponential
geodesic map i. e. $exp_{x}(u)=\gamma (1), u\in T_{x}(\mfd{M}),$ where
$\gamma (t)$ is the geodesic starting at $x$ with the initial
vector $u$ : $\gamma (0)=x , \frac{d\gamma (0)}{dt}=u.$
We denote by $\operatorname{inj}$ the largest real number $r$ such that $exp_x$ is a diffeomorphism of a suitable open neighborhood of $0$ in $T_x \mfd{M}$ onto $B(x,\rho)$, for all $\rho< r$ and $x \in \mfd{M}$. Thus for every choice of an orthonormal basis
(with respect to the inner
 product defined by $g$) of $T_{x}(\mfd{M})$ the exponential map
  $exp$ defines a
 coordinate system on $B(x,\rho)$ which is called {\it geodesic}.
The volume of the ball $B(x,\rho)$ will be denoted
by $|B(x,\rho)|.$
Throughout the paper we will consider only geodesic coordinate
 systems.

We will consider only Riemannian manifolds of bounded geometry.
 Let us recall that a manifold has bounded geometry if \\
(a) \textit {$\mfd{M}$ is complete and connected;}\\
(b)  \textit {the injectivity radius $\operatorname{inj}(\mfd{M})$ is positive;}\\
(c)  \textit
{for any $\rho\leq \operatorname{inj}(\mfd{M})$, any $k \geq 0$
there exists a constant $C(\rho,k)$ such that }
$$
\sup _{z \in \vartheta_y^{-1}(B(x,\rho)\cap B(y,\rho))}\sup_{|\alpha|\leq k}|\partial
^{\alpha}(\vartheta_{x}^{-1}\vartheta_{y}) (z)|\leq C(\rho, k),
$$
for all geodesic
coordinate systems  $\vartheta_{x}: T_{x}(\mfd{M})\rightarrow B(x,\rho),
\vartheta_{y}:T_{y}(\mfd{M})\rightarrow B(y,\rho)$.
% was B(x,\rho) %hgfei

Examples of manifolds of bounded geometry are:  compact Riemannian manifolds,  all Lie
groups with left (right) invariant Riemannian structure and their homogeneous manifolds, covering spaces  of all compact manifolds, bounded domains in $\mathbb{R}^{n}$ with smooth boundaries.

We will also need the following condition:

(d) \textit {The Riemannian measure fulfills the local doubling property  i.e. there exists a constant $C(\mfd{M})$ such that for any sufficiently small $\rho<\operatorname{inj}(\mfd{M})$ and any $0<\sigma<\lambda < \rho$ the following inequality
(local doubling property) holds true}
\begin{equation}\label{LDP}
|B(x,\lambda)|\leq C(\mfd{M})\left(\lambda/\sigma\right)^{n}|B(x,\sigma)|,\>\>n=dim\>\mfd{M}.
\end{equation}

Note that  the Bishop-Gromov Comparison Theorem implies
(see \cite{H}) that this condition is satisfied whenever
the   Ricci curvature  $Ric$ is bounded  from below.
More precisely, if
\begin{equation}\label{Ric}
Ric\geq-kg, \>\>\>k\geq 0
\end{equation}
the {\it local doubling property}  (d) is satisfied for
 $0 <\sigma<\lambda< \delta< \operatorname{inj}(\mfd{M})$:
\begin{equation}\label{BG}
|B(x,\rho)|\leq\left(\rho/\sigma\right)^{n}e^{(k
\delta(n-1))^{1/2}} |B(x,\sigma)|,\>n=dim\>\mfd{M}.
\end{equation}

\begin{lemma}\label{cover}(\cite{Pes04b}, \cite{Pes13b})
If $\mfd{M}$ has bounded geometry and condition (\ref{Ric}) holds, then there exists a natural number $N_{\mfd{M}}$ such that for any
$0<r<\operatorname{inj}\>(\mfd{M})$ there exists a set of points
$M_{r}=\{x_{i}\}$ with the following properties
\begin{enumerate}
\item  \textit {the balls $B(x_{i}, r/4)$ are disjoint,}
\item  \textit {the balls $B(x_{i}, r/2)$ form a cover of $\mfd{M}$,}
\item  \textit {the height of the cover by the balls
  $B(x_{i},\ r)$ is at most  $N_{\mfd{M}}.$}
\end{enumerate}
\end{lemma}

The important feature of this Lemma is the claim that height  $N_{\mfd{M}}$ is independent on $r$ for $0<r<\operatorname{inj}\>(\mfd{M})$.

\begin{definition}\label{lattice}
Any set of points $\{x_{i}\}\in \mfd{M}$ which satisfies the above properties
(1)-(3)  will be denoted as $\mfd{M}_{r}$ and called $(r,\>N_{\mfd{M}})$-lattice of $\mfd{M}$.
\end{definition}

 To construct Sobolev spaces $W_{p}^{k}(\mfd{M}), \>\>
k\in \mathbb{N},$ we fix a $\rho$-lattice $\mfd{M}_{\rho}=\{y_{\nu}\}, \>\>0< \rho< \operatorname{inj} (\mfd{M})$ and introduce a partition of unity  ${\varphi_{\nu}}$ that is
{\it subordinate to the family} $\{B(y_{\nu}, \rho/2)\}$ and has the
following properties:
\begin{enumerate}
\item $\varphi_{\nu}\in C_{0}^{\infty} B(y_{\nu}, \rho/2 ),$

\item $\sup_{x}\sup_{|\alpha|\leq k}|\varphi_{\nu}^{(\alpha
)}(x)|\leq C(k), $ in geodesic  coordinates,
 where $ C(k) $ is independent of $\nu$ for
every $k $ .
\end{enumerate}
Such a partition is called  {\it bounded uniform partition of unity}, for
short a BUPU, or more precisely a $(\rho, N_{\mfd{M}})$-BUPU.
We introduce the  Sobolev space $W_{p}^{k}(\mfd{M}), k\in \mathbb{N},1 \le p \le \infty$ as the set of all $f\in L_{p}({\bf M})$ for which

\begin{equation}\label{Sob}
\|f\|_{W_{p}^{k}(\mfd{M})}=\left(\sum_{\nu}\|\varphi_{\nu}f\|^{p}
_{W_{p}^{k}(\mathbb{R}^{n})}\right) ^{1/p}<\infty.
\end{equation}
Here the norm $\|\varphi_{\nu}f\|^{p}
_{W_{p}^{k}(\mathbb{R}^{n})}$ is to be understood as the Sobolev norm of a pullback of $\varphi_{\nu} f$ to $\mathbb{R}^n$ via a geodesic coordinate system on $B(y_\nu, \rho/2)$.
A way to introduce Besov spaces $\mathcal{B}_{p,q}^{\alpha}({\bf M})$ is by using Peetre's $K$-functional
\begin{equation}\label{Besov-1}
\mathcal{B}_{p, q}^{\alpha }({\bf M})=(L_{p}(\bold{M}),W^{r}_{p}(\bold{M}))^{K}_{\alpha/r,q},\ 0<\alpha<r\in \mathbb{N},\
1\leq p\leq \infty,\>\>\>0<q\leq \infty.
\end{equation}
The fact that one can use in (\ref{Besov-1}) any natural $r>\alpha$ is well known \cite{BB}, \cite{T3}.

Note that \cite{T1}-\cite{T3} gives an equivalent definition
for ${\bf M}$  compact:
\begin{equation}\label{Besov-2}
\|f\|_{\mathcal{B}_{p, q}^{\alpha }({\bf M})}=\left(\sum_{\nu}\|\varphi_{\nu}f\|^{p}
_{\mathcal{B}_{p,q}^{\alpha}(\mathbb{R}^{n})}\right) ^{1/p}<\infty.
\end{equation}
However, due to the lack of the Localization Principle (see \cite{T3}) the later definition cannot be used in the case of non-compact manifolds.

We will explore second-order differential elliptic operators which are self-adjoint and non-negative  definite in the corresponding space $L_{2}(\mfd{M})$.
The best known example of such an operator is the Laplace-Beltrami which is given in a local coordinate system  by the formula
$$
L
f=-\sum_{m,k}\frac{1}{\sqrt{det(g_{ij})}}\partial_{m}\left(\sqrt{det(g_{ij})}
g^{mk}\partial_{k}f\right)
$$
where the $g_{ij}$ denote the components of the metric tensor,$\>\>det(g_{ij})$ is the determinant of the matrix $(g_{ij})$, and $\>\>g^{mk}$ denote the components of the matrix inverse to $(g_{ij})$.
The  Laplace-Beltrami  is a self-adjoint positive definite
operator in the corresponding space $L_{2}(\mfd{M})$
constructed  from $g$. The domains of the powers
 $L^{s/2}$ coincide with the Sobolev spaces
$H^{s}(\mfd{M})=W_{2}^{s}(\mfd{M})$ , for $s \in \mathbb{R}$.

\subsection{Compact Riemannian manifolds}

 In this section we consider only compact Riemannian manifolds.
Let now $L$ be a smooth, self-adjoint, non-negative, second order elliptic differential operator on
 % ${\bf M}$. We will assume that
 % $L$ is self-adjoint and non-negative in
the space $L_{2}(\mathbf{M})$, over a compact Riemannian manifold.
The spectrum of the positive square root $\sqrt{L}$  operator is given by a sequence
\[ 0=\lambda_{0}<\lambda_{1}\leq \lambda_{2}\leq~... \]
approaching infinity. Let
$u_{0}, u_{1}, u_{2}, ...$ be a corresponding
complete system of real-valued orthonormal eigenfunctions, and let
$ \mathbf{E}_{\omega}\SLB,\ \omega>0,$ be the span of all
eigenfunctions of $\sqrt{L}$ whose corresponding eigenvalues
are not greater than $\omega$. Clearly $\mathbf{E}_{\omega} =\PWoL$,
and these subspaces are finite-dimensional,
and contained in $C^\infty(\mfd {M}) \subset L_p(\mfd{M}),\>\>\>1\leq p\leq \infty$. Since the
operator $L$ is of order two, the dimension
$\mathcal{N}_{\omega}$ of the space ${\mathbf E}_{\omega}\SLB$ is
given asymptotically by Weyl's formula \cite{Sog}, which says, in sharp form:
for some $c > 0$, and $n = dim {\bf M}$ one has
\begin{equation}
\label{Weyl}
\mathcal{N}_{\omega}(L) = c\omega^{n} + O(\omega^{(n-1)}).
\vspace{.3cm}
\end{equation}
% where $s=dim {\bf M}$.
Since $\mathcal{N}_{\lambda_l} = l+1$, we conclude that, for some constants $c_1, c_2 > 0$,
\begin{equation}
\label{lamest}
c_1 \, l^{1/n} \leq \lambda_l \leq c_2 \, l^{1/n},
\quad l \geq 0.
\end{equation}
Since $L^m u_l = \lambda_l^{2m} u_l$, and $L^m$ is an elliptic differential
operator of degree $2m$, Sobolev's lemma, combined with the last fact, implies that
for any integer $k \geq 0$, there exist $C_k$, and
$ \nu_k > 0$ such that
\begin{equation}
\label{ulest}
\|u_l\|_{C^k({\bf M})} \leq C_k (l+1)^{\nu_k}.
\end{equation}

\subsection {Compact homogeneous manifolds}

The most complete results will be obtained for compact homogeneous manifolds.

A {\it compact homogeneous manifold} $\mfd{M}$ is a
$C^{\infty}$-compact manifold  on which a compact
Lie group $G,\>\>\>dim\>G=d,$ acts transitively. In this case $\mfd{M}$ is necessary of the form $G/K$,
where $K$ is a closed subgroup of $G$. The notation $L_{2}(\mfd{M})$ is used for the usual Hilbert spaces,  with  invariant measure  $dx$ on $\mfd{M}$.

The Lie algebra $\textbf{g}$ of a compact Lie group $G$ %hgfei
is then a direct sum
$\textbf{g}=\textbf{a}+[\textbf{g},\textbf{g}]$, where
$\textbf{a}$ is the center of $\textbf{g}$, and
$[\textbf{g},\textbf{g}]$ is a semi-simple algebra. Let $Q$ be a
positive-definite quadratic form on $\textbf{g}$ which, on
$[\textbf{g},\textbf{g}]$, is opposite to the Killing form. Let
$X_{1},...,X_{d}$ be a basis of
$\textbf{g}$, which is orthonormal with respect to $Q$.
 Since the form $Q$ is $Ad(G)$-invariant, the operator
\begin{equation}\label{Casimir}
-X_{1}^{2}-X_{2}^{2}-\    ... -X_{d}^{2},    \ d=dim\ G
\end{equation}
is a bi-invariant operator on $G$, which is known as the {\it Casimir operator}.
This implies in particular that the corresponding operator on $L_{2}(\mfd{M})$,
\begin{equation}\label{Casimir-Image}
\McL=-D_{1}^{2}- D_{2}^{2}- ...- D_{d}^{2}, \>\>\>
       D_{j}=D_{X_{j}}, \        d=dim \ G,
\end{equation}
commutes with all operators $D_{j}=D_{X_{j}}$.
The operator $\McL$, which is usually called the {\it Laplace operator}, is
the image of the Casimir operator under the differential of the quasi-regular representation in $L_{2}(\mfd{M})$. It is important to realize that in general, the operator $\McL$ is not necessarily the  Laplace-Beltrami operator of the natural  invariant metric on $\mfd{M}$. But it coincides with this operator at least in the following cases:
1) If $\mfd{M}$ is a $d$-dimensional torus, 2) If the manifold $\mfd{M}$ is itself a compact  semi-simple Lie group group $G$ (\cite{H2}),Chap. 3).
If $\mfd{M}=G/K$ is a compact symmetric space of   rank one (\cite{H2}).

\subsection{An example:  the  sphere $\mathbb{S}^{d}$}

We will specify the general setup in the case of standard unit sphere.
Let us write
$$
\mathbb{S}^{d}=\left\{x\in \mathbb{R}^{d+1}: \|x\|=1\right\}.
$$
We denote the space of spherical harmonics of degree $l$
by the symbol $\mathcal{P}_{l}$. They are the restrictions
of harmonic homogeneous polynomials of degree $l$
in $\mathbb{R}^{d}$ to $\mathbb{S}^{d}$.
The Laplace-Beltrami operator $\Delta_{\mathbb{S}}$ on $\mathbb{S}^{d}$  is the pullback % restriction  %hgfei: OK?
of the regular Laplace operator $\Delta$ in $\mathbb{R}^{d}$,
given by
 $$
\Delta_{\mathbb{S}}f(x)=\Delta \widetilde{f}(x),\>\>x\in \mathbb{S}^{d},
$$
where $\widetilde{f}(x)$ is the homogeneous extension of $f$: $\>\>\widetilde{f}(x)=f\left(x/\|x\|\right)$. Another way to compute $\Delta_{\mathbb{S}}f(x)$ is to express both $\Delta_{\mathbb{S}}$ and $f$ in a spherical coordinate system.

Each $\mathcal{P}_{l}$ is the eigenspace of  $\Delta_{\mathbb{S}}$ that corresponds to the eigenvalue $-l(l+d-1)$.
This space has dimension $n_{d}$, given by
$$
n_{d}(l)=(d+2l-1)\frac{(d+l-2)!}{l!(d-1)!}.
$$
An orthonormal basis for the eigenspace
$\mathcal{P}_{l},\>\>l=0, 1, 2,..., $
will be denote by  $\mathcal{Y}_{n,l},\>\>n=1,...,n_{d}(l)$.
% be an orthonormal basis in $\mathcal{P}_{l}$.%
%{\sc attention, it was $P_n$ in the original version! Hans}%

Let $e_{1},...,e_{d+1}$ be the standard orthonormal basis in $\mathbb{R}^{d+1}$.  % ???  Hans%
Writing  $SO(d+1)$ and $SO(d)$ for the groups of rotations of $\mathbb{R}^{d+1}$ and  $\mathbb{R}^{d}$ respectively we have $\mathbb{S}^{d}=SO(d+1)/SO(d)$.
On $\mathbb{S}^{d}$ we consider the vector fields
$$
X_{i,j}=x_{j}\partial_{x_{i}}-x_{i}\partial_{x_{j}},\>\>\>i<j,
$$
which are generators of one-parameter groups of rotations   $\exp tX_{i,j}\in SO(d+1)$ in the plane $(x_{i}, x_{j})$. These groups are defined by the formulas
$$
\exp \tau X_{i,j}\cdot (x_{1},...,x_{d+1})=(x_{1},...,x_{i}\cos \tau -x_{j}\sin \tau ,..., x_{i}\sin \tau +x_{j}\cos \tau ,..., x_{d+1}).
$$
Let $e^{\tau X_{i,j}}$ be a one-parameter group which is a representation of $\exp \tau X_{i,j}$ in a space $L_{p}(\mathbb{S}^{d})$. It acts on $f\in L_{p}(\mathbb{S}^{d})$
 by the following formula
$$
e^{\tau X_{i,j}}f(x_{1},...,x_{d+1})=f(x_{1},...,x_{i}\cos \tau -x_{j}\sin \tau ,..., x_{i}\sin \tau +x_{j}\cos \tau ,..., x_{d+1}).
$$
The Laplace-Beltrami operator $\Delta_{\mathbb{S}}$ can be identified with an operator in $L_{p}(\mathbb{S}^{d})$, given by the formula
 $$
\Delta_{\mathbb{S}}=L=-\sum_{(i,j)}X_{i,j}^{2}.
$$
Note, that $L\mathcal{Y}_{n,l}=-l(l+d-1)\mathcal{Y}_{n,l}$.
Since the vector fields $X_{i,j}$ generate tangent space at every point of $\mathbb{S}^{d}$ the operator $L$ is elliptic and domains of its natural powers coincide with the regular Sobolev spaces $W_{p}^{k}(\mathbb{S}^{d})$.  Clearly,  the norm of $B_{2,2}^{\alpha}(\mathbb{S}^{d})$ is equivalent to
\begin{equation}\label{norm1}
\left(\sum_{l=0}^{\infty}
\sum_{n=1}^{n_{d}(l)}(l+1)^{2\alpha}|c_{n,l}(f)|^{2}\right)^{1/2},
\>\>\>c_{n,l}(f)=\int_{\mathbb{S}^{d}} \negthinspace
f \cdot \mathcal{Y}_{n,l} \, ds ,\>\>\>f\in L_{2}(\mathbb{S}^{d}).
\end{equation}
Another description of the
 Besov spaces $B_{p,q}^{\alpha}(\mathbb{S}^{d})$ can be given
using the modulus of continuity associated to the one-parameter groups $e^{\tau X_{i,j}}$ (see Section \ref{ModCont}).

  \subsection{Bounded domains with smooth boundaries}

Let $\Omega\subset \mathbb{R}^{n}$ be a bounded domain with a smooth boundary  $\Gamma$, assumed to be a smooth $(d-1)$-dimensional oriented manifold.  Let $\overline{\Omega}=\Omega\cup \Gamma$ and $L_{2}(\Omega)$ be  the  space of functions  square-integrable with respect to Lebesgue  measure $dx=dx_{1}...dx_{n}$. If $k$ is a natural number the notation $H^{k}(\Omega)$ will be  used for  the Sobolev space of distributions on $\Omega$ (see \cite{LMag} for a precise definition) 
with the norm
$$
 \|f\|_{H^{k
 }(\Omega)}=\left(\|f\|^{2}+\sum _{1\leq |\alpha |\leq k}\|\partial^{\alpha} f\|^{2}\right)^{1/2}
 $$

where $\alpha=(\alpha_{1},...,\alpha_{d})$ is a natural vector and $\partial^{\alpha}$ is a mixed partial derivative
$$
\left(\frac{\partial}{\partial x_{1}}\right)^{\alpha_{1}}...\left(\frac{\partial}{\partial x_{n}}\right)^{\alpha_{n}}.
$$
Under our assumptions the space $C^{\infty}_{0}(\overline{\Omega})$ of infinitely smooth functions with support in $\overline{\Omega}$  is dense
in $H^{k}(\Omega)$. The closure of the space   $C_{0}^{\infty}(\Omega)$ of smooth functions with support in $\Omega$ in $H^{k}(\Omega)$
is denoted by $H_{0}^{k}(\Omega)$.

 Since $\Gamma$ can be treated as a smooth Riemannian manifold one can introduce a Sobolev scale of spaces $H^{s}(\Gamma),\>\> s\in \mathbb{R},$ as, for example, the domains of the Laplace-Beltrami operator $\McL$ of a Riemannian metric on $\Gamma$.

The trace theorem provides a continuous
and surjective trace operator
 $$
 \gamma: H^{s}(\Omega)\rightarrow H^{s-1/2}(\Gamma),\>\>s>1/2,
 $$
 such that for all functions $f \in H^{s}(\Omega)$ which are smooth up to the boundary the value $\gamma (f)$ %hgfei brackets
is simply a restriction of $f$ to $\Gamma$.

 One considers  a strictly elliptic self-adjoint positive definite  operator $L$ generated by an expression
 \begin{equation}\label{Op}
 Lf=-\sum_{k,i=1}^{d}\partial_{_{x_{k}}}
 \left(a_{k, i}(x)\partial_{x_{i}}f\right),
\end{equation}
with coefficients in $C^{\infty}(\Omega)$ where the matrix $(a_{j,k}(x))$ is real, symmetric  and positive definite on $\overline{\Omega}$.
 The operator $L$ is defined as the Friedrichs extension of $L$, initially defined on $C_{0}^{\infty}(\Omega)$, to the set of all functions $f$ in $H^{2}(\Omega)$ with constraint $\gamma f=0$. The Green formula implies that this operator is self-adjoint.  The domain of its positive square root $\sqrt{L}$ is the set of all functions $f$ in $H^{1}(\Omega)$ for which $\gamma f=0$.

Thus, one obtains a self-adjoint positive definite operator
in the Hilbert space $L_{2}(\Omega)$ with a discrete spectrum $0<\lambda_{1}\leq \lambda_{2}\leq... $, % which goes to infinity.
with $\lim_{n \to \infty} \lambda_n = +\infty$.

 An important example of such a situation is the Dirichlet Laplacian  on the unit ball in $\mathbb{R}^{n}$.
In  spherical coordinates $(r,\vartheta),\>\>\vartheta\in \mathbb{S}^{n-1},$ one has
 \begin{equation}
Lf= \partial^{2}_{r}f+\frac{n-1}{r}\partial_{r}f-\frac{1}{r^{2}}\Delta_{\mathbb{S}^{n-1}} f,
 \end{equation}
with the boundary condition
 \begin{equation}\label{bc}
 f|_{r=1}=0,
 \end{equation}
 where $\Delta_{\mathbb{S}^{n-1}}$ is the Laplace-Beltrami operator on the unit sphere $\mathbb{S}^{n-1}$ in $\mathbb{R}^{n}$.

 It is known that the eigenvalues of such an operator are
 given by the formula
  $$
 \lambda=j^{2}_{m+\frac{n-2}{2},\>l},
 $$
 where $j_{\nu,\>l}$ is the $l$-th positive zero of the Bessel function of first kind $J_{\nu}$ of order $\nu$   and the corresponding eigenfunctions are of the   form
 \begin{equation}\label{EF}
 u_{m,\>k,\>l}=c_{m,k,l}r^{-\frac{n-2}{2}}J_{m+\frac{n-2}{2}}\left(j_{m+\frac{n-2}{2}, \>l}\>\>r\right)Y^{k}_{m,n}(\vartheta),
\end{equation}
with $m=0,1,...,\>\>1\leq k\leq k_{m,n},\>\> l=1,2,....$.  The constants $c_{m,\>k,\>l}$  are chosen to normalize the  functions $u_{m,\>k,\>l}$ with respect to $\| \cdot \|_2$.

 \subsection{The Poincar\'{e} hyperbolic upper half-plane}

To  illustrate   our results about sampling and frames on non-compact symmetric spaces we  will  use the hyperbolic
plane in its upper half-plane realization.

  Let $G=SL(2,\mathbb{R})$ be the special linear group of all
$2\times 2$ real matrices with determinant 1,  and let $K=SO(2)$ denote the
group of all rotations of $\mathbb{R}^{2}$. The factor
$\mathbb{H}=G/K$ is known as the 2-dimensional hyperbolic space
and can be described in many different ways. In the present paper
we consider the realization of $\mathbb{H}$ which is called
Poincar\'{e} upper half-plane (see \cite{H2}, \cite{T}).

As a Riemannian manifold $\mathbb{H}$ is identified with the
regular upper half-plane of the complex plane
$$
\mathbb{H}=\{x+iy|x,y\in \mathbb{R}, y>0\}
$$
with a new  Riemannian metric
$$
ds^{2}=y^{-2}(dx^{2}+dy^{2})
$$
and corresponding Riemannian measure
 $$
  d\mu=y^{-2}dxdy.
   $$
   If we
define the action of $\sigma\in G$ on $z\in \mathbb{H}$ as a
fractional linear transformation
$$
\sigma\cdot z=(az+b)/(cz+d),
$$
then the metric $ds^{2}$ and the measure $d\mu$  are invariant
under the action of $G$ on $\mathbb{H}$. The point $i=\sqrt{-1}\in
\mathbb{H}$ is invariant for all $\sigma\in K$. The Haar measure
$dg$ on $G$ can be normalizes in a way that the  following
important formula holds true
$$
\int_{\mathbb{H}}f(z)y^{-2}dxdy=\int_{G}f(g\cdot i) d g.
$$
In the corresponding Hilbert space % of square integrable functions
$L_{2}(\mathbb{H})$ with the inner product
$$
\langle f,h \rangle =\int_{\mathbb{H}}f (x,y)
\overline{h(x,y)} \, \, y^{-2}\, dxdy
$$
we consider the Laplace-Beltrami operator
 $$
 \Delta=-y^{2}\left(\partial_{x}^{2}+\partial_{y}^{2}\right)
 $$
of the metric $ds^{2}$. The operator
 $\Delta$, acting on $L_{2}(\mathbb{H})=L_{2}(\mathbb{H},d\mu)$
 is initially defined on $C_{0}^{\infty}(\mathbb{H})$ and has a self-adjoint closure in $L_{2}(\mathbb{H})$.

The Helgason-Fourier  transform
of $f$ for $\lambda\in \mathbb{C}, \varphi\in (0,2\pi]$, is defined by
the formula
$$
\hat{f}(\lambda,\varphi)=\int_{\mathbb{H}}f(z)\overline{Im(k_{\varphi}z)^{\lambda}}
y^{-2}dxdy,
$$
where $k_{\varphi}\in SO(2)$ is the rotation of $\mathbb{R}^{2}$
by angle $\varphi$.
We have the following inversion formula for all $f\in
C_{0}^{\infty}(\mathbb{H})$
$$
f(z)=(8\pi^{2})^{-1} \int_{\lambda\in\mathbb{R}}\int_{0}^{2\pi}
\hat{f}(i\lambda+1/2,\varphi)Im(k_{\varphi}z)^{i\lambda+1/2}\lambda \tanh \pi \lambda
d\varphi d\lambda.
$$
The Plancherel Theorem states that the map
$f\rightarrow \hat{f}$ can be extended to an isometry of
$L_{2}(\mathbb{H})$ with respect to the invariant measure $d\mu $ onto
$L_{2}(\mathbb{R}\times (0,2\pi ])$ with respect to the measure
$$
\frac{1}{8\pi^{2}} \lambda \tanh \pi\lambda d\lambda d\varphi.
$$
If $f$ is a function on $\mathbb{H}$ and $\varphi$ is a
$K=SO(2)$-invariant function on $\mathbb{H}$, their convolution is
defined by the formula
$$
f\ast \varphi(g\cdot i)=\int_{SL(2,\mathbb{R})}f(gu^{-1}\cdot
i)\varphi(u)du, ~i=\sqrt{-1},
$$
where $du$ is the Haar measure on $SL(2,\mathbb{R})$.
For the  Helgason-Fourier transform  one has:
$$
\widehat{f\ast\varphi}=\hat{f}\cdot\hat{\varphi}.
$$
The following formula holds true
\begin{equation}
\widehat{\Delta f}=\left(\lambda^{2}+\frac{1}{4}\right)\hat{f}.
\end{equation}
The one parameter group of operators $e^{it\Delta }$ acts on  functions via the formula
$$
e^{it\Delta }f(z)=f\ast G_{t},
$$
where
$$
G_{t}(ke^{-r}i)=(4\pi)^{-1}\int_{s\in
\mathbb{R}}e^{-i(s^{2}+1/4)t}P_{is-1/2}(\cosh \>r)s  \tanh\>\pi s ds,
$$
here $k\in SO(2),$ $r$ is  the geodesic distance, $ke^{-r} i$ is
representation of points of $\mathbb{H}$ in the geodesic polar
coordinate system on $\mathbb{H}$, and $P_{is -1/2}$ is the
associated Legendre function.
In these terms our Theorem \ref{PW proprties} takes the following form.
\begin{theorem}
For  $f\in L_{2}\left(\mathbb{H}, \>d\mu\right)$ the following conditions are equivalent.
\begin{enumerate}
\item  $f$  belongs to the space $PW_{\omega}(\Delta)$.
\item For every
$\sigma\in \mathbb{R}$ the following Bernstein inequality holds true
\begin{equation} \|\Delta^{\sigma} f\|\leq
\left(\omega^{2}+\frac{1}{4}\right)^{\sigma}\|f\|,
 \end{equation}
where $\|f\|$ means the $L_{2}(\mathbb{H})$ norm of $f$.

\item   For every $g\in L_{2}(\mathbb{H}, d\mu)$ the function
$$
t\rightarrow \left<f\ast G_{t},g\right>=\int_{\mathbb{H}}f\ast
G_{t}\overline{g}\, d\mu
$$
is an  entire function of the exponential type
$\omega^{2}+\frac{1}{4}$ 
  bounded on the real line $\mathbb{R}$.
\end{enumerate}
\end{theorem}

\section{Generalized Shannon-type sampling in Paley-Wiener spaces and frames  in $L_{2}$-spaces on Riemannian manifolds of bounded geometry}\label{SFonM}

In this section we treat both compact and non-compact Riemannian manifolds of bounded geometry whose Ricci curvature is bounded from below (see section 3.1). The material in this section is based on \cite{Pes00}, \cite{Pes04b}.

\subsection{A Shannon-type sampling theorem in Paley-Wiener spaces on compact and non-compact Riemannian manifolds}

The most important fact for our development is an analogue of the Shannon's Sampling Theorem for Riemannian manifolds of bounded geometry which first appeared in \cite{Pes00} and was further developed in   \cite{FP04}, \cite{FP05}, \cite{Pes04a}-\cite{Pes09b}, (for subelliptic versions see  \cite{HF},  \cite{FuhrG}, \cite{Pes98a}, \cite{Pes98b}).

Let $\mfd{M}_{r}$  be an $r$-lattice and let $\{B(x_{k},r)\}$ be an associated family of balls that satisfy the  properties of Definition \ref{lattice}.
We define
$$
U_{1}=B(x_{1}, r/2)\setminus \cup_{i,\>i\neq 1}B(x_{i}, r/4),
$$
and
\begin{equation}\label{disjointcover}
U_{k}=B(x_{k}, r/2)\setminus \left(\cup_{j<k}U_{j}\cup_{i,\>i\neq k}B(x_{i}, r/4)\right).
\end{equation}
It is easy to verify the following statement.
\begin{lemma} The sets $\left\{U_{k}\right\}$ form a disjoint measurable cover of $\mfd{M}$ and
\begin{equation}\label{disjcover}
B(x_{k}, r/4)\subset U_{k}\subset B(x_{k}, r/2).
\end{equation}
\end{lemma}

We assume that on every $U_{k}$
a strictly positive measure $\mu_{k}$, supported in $U_{k}$,
is given.
We consider the following distribution on
$C_{0}^{\infty}(B(x_{k},r)),$
\begin{equation}\label{distrib}
\mathcal{M}_{k}(f)=\frac{1}{|U_{k}|_{\mu_{k}}}\int_{U_{k}}fd\mu_{k},\>\>\> \quad \mbox{with} \quad
|U_{k}|_{\mu_{k}}=\int_{U_{k}}d\mu_{k}, \>\>\>\>  f\in C_{0}^{\infty}(B(x_{k},r)).
\end{equation}
 As a compactly
supported distribution of order zero it has a unique continuous
extension to a function on the space $C^{\infty}(B(x_{k}, r))$.

We say that a family $\mathcal{M}=\{\mathcal{M}_{k}\}$ is uniformly bounded, if there exists a positive constant $C_{\mathcal{M}}$
such that % for all $k$
\begin{equation}\label{uni-bound}
\mathcal{M}_{k}(f)\leq C_{\mathcal{M}}\sup_{x\in B(x_{k}, r)}|f(x)|,\>\>\>f\in C^{\infty}(B(x_{k}, r))
\quad \mbox{for all}\,\, k.
 \end{equation}
\bigskip

Some examples of distributions which are of particular
interest are the following.
\begin{enumerate}
\item  Weighted Dirac measures  $\mathcal{M}_{k}(f)=
\muk_{k}\delta_{x_{k}}(f),\>\>\>x_{k}\in U_{k},\>\>\muk_{k} >0.$
\item  Finite or infinite sequences of weighted Dirac measures
$$
\mathcal{M}_{k}(f)=
\sum_{x_{k, l}\in  U_{k}}\muk_{k, l}\delta_{x_{k, l}}(f),\>\>\>\muk_{k,l}>0.
$$
\item  $d\mu_{k}$ is a "surface" measure on a submanifold contained in   $U_{k}$ .
\item $d\mu_{k}$ is the
restriction to $U_{k}$ of the Riemannian measure $dx$ on $\mfd{M}$.
\end{enumerate}
\begin{lemma}\label{LPI}(Local Poincar\'{e}-type inequality \cite{Pes04b})
For    $m>n/2$ ($n = dim {\bf M}$)
there exist positive constants $C=C(\mfd{M}, m)>0, \>\>r(\mfd{M},m)>0,$
 such that for any $(r, N_{\mfd{M}})$-lattice $M_{r}$ with $r<r(\mfd{M},m)$
and any associated family of functional
$\mathcal{M}_{k}$  the following inequality
holds true for all $f\in H^{m}(\mfd{M})$:
\begin{equation}\label{GPII}
\|(\varphi_{\nu}f)-\mathcal{M}_{k}((\varphi_{\nu}f))\|_{L_{2}(U_{k})}\leq
C\sum_{1\leq |\alpha|\leq
m}r^{|\alpha|}\|\partial^{\alpha}(\varphi_{\nu}f)\|_{L_{2}(B(x_{k}, r))},
\end{equation}
where for % a natural vector
any multi-index
$\alpha=(\alpha_{1},...,\alpha_{n}) \in \mathbb{N}^n$
the symbol $\partial^{\alpha}f$ stands
 for  a partial derivative $\partial_{x_{1}}^{\alpha_{1}}...\partial_{x_{n}}^{\alpha_{n}}$ in a geodesic  coordinate system.
\end{lemma}

We introduce the following set of functionals
\begin{equation}\label{A-functionals}
\mathcal{A}_{k}(f)=\sqrt{|U_{k}|}\mathcal{M}_{k}(f)=
\frac{\sqrt{|U_{k}|}}{|U_{k}|_{\mu_{k}}}\int_{U_{k}}f(x)d\mu_{k},\>\>\>|U_{k}|=\int_{U_{k}}dx,\>\>\>f\in L_{2}(\mfd{M}),
\end{equation}
where $dx$ is the Riemann measure on ${\bf M}$.

For the further development it will be important to
associate with these functionals families of  band-limited
functions. In fact, the restriction of the functionals
above to the closed subspaces $\PWoL \subset L_{2}(\mfd{M})$
can be realized by scalar products, via the Riesz representation
theorem: There are uniquely determined functions
$\phi_{\omega, k} \in \PWoL$  such that
 \begin{equation}\label{phi}
 \left<f,\phi_{\omega, k}\right>=
 \mathcal{A}_{\omega, k}(f),\>\>\>f\in \PWoL.
 \end{equation}

It is convenient to introduce the following definition.

\begin{definition}\label{aff}
For a given $r$-lattice $\mfd{M}_{r}$, let $\{U_{k}\}$ be the disjoint cover constructed in (\ref{disjointcover}), and functionals $\mathcal{M}_{k}$, $\mathcal{A}_{k}$ defined as in (\ref{distrib}) and (\ref{A-functionals}) respectively.

Given an $\omega>0$ the system of functions $\{\phi_{\omega, k}\}$ defined in (\ref{phi}) will be called the family of functions associated with the pair $\left(\mfd{M}_{r},\>\PWoL\right)$.
\end{definition}

\begin{lemma}\label{GPI}
(Global Poincar\'{e}-type inequality)
For any   $0<\delta<1$ and $m>\frac{1}{2} dim\ {\bf M}, $ there exist constants $c=c(\mfd{M}),\>C=C(\mfd{M},m),$  such that  following inequality holds true for any $(r, N_{\mfd{M}})$-lattice with $r\leq c\delta$ and any $ H^{m}(\mfd{M})$:
\begin{equation}\label{gpi}
(1-\delta/2)\|f\|^{2} \leq
\sum_{k}|\mathcal{A}_{k}(f)|^{2}+C\delta^{-1}r^{2m}\|L^{m/2}f\|^{2}.
\end{equation}
\end{lemma}

\begin{proof} We sketch a proof for the case of a manifold without boundary (compact or non-compact). In the presence of a boundary more care is required around the boundary.  (see \cite{Pes15c}).
Applying Lemma \ref{LPI} and the inequality
\begin{equation}\label{ineq-1}
(1-\alpha)|A|^{2}\leq \frac{1}{\alpha}|A-B|^{2}+|B|^{2},\>\>0<\alpha<1,
\end{equation}
we obtain
\begin{equation}\label{ineq-2}
(1-\delta/3) \, \|\varphi_{\nu}f\|^{2}_{L_{2}(U_{k})}\leq
3/\delta\sum_{k}\|f-\mathcal{M}_{k}(\varphi_{\nu}f)\|^{2} _{L_{2}(U_{k})}+
\sum_{k}|U_{k}||\mathcal{M}_{k}(\varphi_{\nu}f)|^{2}.
\end{equation}
From this  and (\ref{GPII}) one obtains
$$
(1-\delta/3) \|f\|^{2}_{L_{2}(\mfd{M})}\leq
\sum_{\nu}\sum_{k}|U_{k}| |\mathcal{M}_{k}(\varphi_{\nu}f)|^{2}+
C_{1}(M,m)\delta^{-1}
\sum_{1\leq j\leq
m}r^{2j}\|f\|^{2}_{H^{j}(\mfd{M})}.
$$
The regularity theorem for the elliptic second-order differential   operator $L$ (see \cite{Hor},  Sec. 17.5)
 \begin{equation}\label{reg}
\|f\|^{2}_{H^{j}(\mfd{M})}\leq b\left(\|f\|^{2}_{2}+\|L^{j/2}f\|^{2}_{2}\right),\>\>f\in \mathcal{D}(L^{m/2}),\>\>b=b(\mfd{M},  j),
\end{equation}
and the  interpolation inequality (see \cite{Hor},  Sec. 17.5)
\begin{equation}\label{interpol}
r^{2j}\|L^{j/2}f\|^{2}_{2}\leq 4a^{m-j}r^{2m}\|L^{m/2}f\|^{2}_{2}
+ca^{-j}\|f\|^{2}_{2}, \>c=c(\mfd{M}, m),
\end{equation}
which holds for any $a,\>r>0, \>0\leq j\leq m$, imply that  there exists a constant   $C=C(\mfd{M},  m)$ such that for any $0<\delta<1$ and $\>\>\>r>0$
$$
(1-\delta/3)\|f\|^{2}_{2}\leq
\sum_{k} |U_{k}||\mathcal{M}_{k}(f)|^{2}+C\left(r^{2}\delta^{-1}\|f\|^{2}_{2}+r^{2m}\delta^{-1}\|L^{m/2}f\|^{2}_{2}\right).
$$
The last inequality shows that if for  a given $0<\delta<1$ and $c=(6C)^{-1/2}$ the value of $r $ is chosen such that
 $r<c\delta$
then we obtain (\ref{gpi}).
The lemma is proved.
\end{proof}

Let us pick $m=n=$dim$\>\mfd{M}$. With this choice  the inequality (\ref{gpi}) is the same as inequality (\ref{A}) for $\rho_{\mathcal{A}}=r$ and $m_{0}=n$.
Note that for functions in $\PWoL$ the Bernstein inequality holds
$$
\|L^{m/2}f\|_{2}\leq \omega^{m}\|f\|_{2},\>\>\>
f\in \PWoL.
$$

Using  property (\ref{uni-bound}) and a Sobolev Embedding Theorem one can verify that condition  (\ref{B1}) is  satisfied.
Thus Theorem \ref{Ss} gives the following version of the Sampling Theorem. It describes the appropriate sampling density
in relation to the `bandwidth' of the $PW$-space (see \cite{Pes00}, \cite{Pes04b}):
\begin{theorem}(Almost Parseval frames in Paley-Wiener spaces) \label{Frame-th}
Let $\mfd{M}$ be a manifold of dimension $n$, with
bounded geometry and Ricci curvature bounded from below.
Then there exists   $c=c(\mfd{M}) >0 $ %uuu change of order
such that one has: Given $r>0$ and $\omega>0$ such that
\begin{equation}\label{rate}
0<r\leq c \> \delta^{1/n} \> \omega^{-1},\>
% \>\mbox{with} \,\, \>n=dim\ {\bf M}, %uuu moved up
\end{equation}
 then any family of functions $\{\phi_{\omega, k}\}$ associated to  the pair $\left(\mfd{M}_{r},\>\PWoL\right)$ (see Definition \ref{aff})
 forms a frame in $\PWoL$, and  %hgfei: not just i.e.!
 the following Plancherel-Polya-type inequalities
 (frame inequalities) hold true:
\begin{equation}\label{frame-ineq-0}
(1-\delta)\|f\|^{2}_{2}\leq \sum_{k}\left |\left<f, \phi_{\omega, k}\right>\right|^{2}\leq\|f\|^{2}_{2},\>\>
% \>0<\delta<1, %hgfei \delta is already given and fixed!!
\>\>f\in \PWoL.
\end{equation}
\end{theorem}

 \subsection{Methods of reconstruction of Paley-Wiener functions}

 For reconstruction of a function from the set of samples one can use besides  dual frames the following methods:

\begin{enumerate}
 \item Reconstruction by variational (polyharmonic) splines on manifolds  \cite{Pes98a}-\cite{Pes09b}.
 \item Reconstruction using iterations  \cite{FP04}, \cite{FP05}.
 \item  Reconstruction by the frame algorithm \cite{Gr}.
 \end{enumerate}

\subsection{Optimality of the number of sampling points on compact manifolds }
\label{subsect:sampling_density}

The material in this section is based on \cite{Pes04b}, \cite{Pes09b}.
Condition  (\ref{rate}) imposes a specific rate of sampling.
It is interesting to note  that this rate is essentially optimal.
Indeed, since $\sqrt{L}$ is a non-negative elliptic pseudodifferential operator of order one  the  Weyl's asymptotic formula \cite{Hor} gives
\begin{equation}
\mathcal{N}_{\omega}\SLB \asymp C~Vol(\mfd{M})\omega^{n},\label{W}
\end{equation}
where $\mathcal{N}_{\omega}\SLB$ is the dimension of the space $\PWoL$ and  $Vol(\mfd{M})$ is the  volume of $\mfd{M}$.
On the other hand,   condition  (\ref{rate}) and the definition of a $r$-lattice imply  that the number of points
in an "optimal" lattice $M_{r}$
% can be  approximately estimated as
is approximately
$$
 card\>M_{r}\sim\frac{Vol(\mfd{M})}{c_{0}^{n}\omega^{-n}}=c\>Vol(\mfd{M})\omega^{n},\>\>\>n=dim\>\mfd{M},
$$
which is consistent with Weyl's formula. Note that the power $n$ in this formula, which is what one would expect from the Shannon sampling theorem in the euclidean setting, results from our use of the spectrum of $\sqrt{L}$ rather than that of $L$.

\subsection{Paley-Wiener almost Parseval frames in $L_{2}$ spaces on compact and non-compact Riemannian manifolds}

We return to the notation of Subsections \ref{subsect:part_unity_freq} and \ref{Hilb}.  Let $\omega_{j} = 2^{j+1}$, for $j \ge 0$. According to  Theorem \ref{Frame-th} for a fixed $0<\delta<1$ there exists a constant $c=c(\mfd{M})$ such that  for
$$
r_{j}=
c\delta^{1/n}\omega_{j}^{-1}=c\delta^{1/n}2^{-j-1},\>\>\>j\in \mathbb{N},
$$
and  any $r_{j}$-lattice $M_{r_{j}}=\{x_{j,k}\},\>\>1\leq k\leq \mathcal{K}_{j},$ the inequalities (\ref{frame-ineq-0}) hold.

For every $j\in \mathbb{N}$ let $\mathcal{M}_{j}=\{\mathcal{M}_{j,k}\}$ be an associated set of distributions described in (\ref{distrib}). %(\ref{uni-separate}).
For the sake of simplicity we now assume that  for every $j\in \mathbb{N}$  the set  of distributions $\mathcal{M}_{j}=\{\mathcal{M}_{j,k}\}$ described in (\ref{distrib}) consists of Dirac measures $\delta_{j,k}$ at points $x_{j,k}$. Let  $\{\mathcal{A}^{j}_{k}\}$ be the corresponding set of functionals  defined in (\ref{A-functionals})  and let  $\phi^{j}_{k}\in \bPW_{\omega_{j}}\SLB = \bPW_{2^{j+1}} \SLB$  be a function such that
\begin{equation}\label{functions-10}
\left<f,\phi^{j}_{k}\right>=\mathcal{A}^{j}_{k}(f),
\end{equation}
for all $f\in \bPW_{2^{j+1}}\SLB.$  If $(F_j)_{j \ge 0}$ is the quadratic partition of unity introduced in Subsection \ref{subsect:part_unity_freq}
then since $F_{j} \SLB f\in \bPW_{2^{j+1}} \SLB$
we have according to (\ref{frame-ineq-0}) the following frame inequalities  for every $j\in \mathbb{Z}$
\begin{equation}
(1-\delta)\left \|F_{j}\ \SLB f\right\|^{2}_{2}\leq
\sum _{ k=1}^{\mathcal{K}_{j}}\left|\left<F_{j}\SLB f, \phi^{j}_{k}\right>\right|^{2}\leq \left\| F_{j}\SLB f\right\|^{2}_{2},\>\>\>f\in L_{2}({\bf M}).
\end{equation}
But since the operator $F_{j}\SLB $ is self-adjoint,
we obtain (via (\ref{norm equality-0})) that for the functions
\begin{equation}\label{frame-functions}
\Phi^{j}_{k}=F_{j}\SLB \phi^{j}_{k}
\end{equation}
which are bandlimited to $ [2^{j-1},\>\>2^{j+1}]$,
the following frame inequalities hold
\begin{equation}\label{frame-ineq}
(1-\delta)\|f\|^{2}_{2}\leq\sum_{j\geq 0}\>\>\sum _{k=1}^{\mathcal{K}_{j}}\left|\left<f, \Phi^{j}_{k}\right>\right|^{2}\leq \|f\|^{2}_{2},\>\>\>\>f\in L_{2}(\mfd{M}).
\end{equation}

To summarize, let us assume that a Riemannian manifold $\mfd{M}$  has bounded geometry and (\ref{Ric}) holds.  Let $c>0$ be a positive constant and $0<\delta<1$.  We consider the following:

 \begin{enumerate}
 \item  a sequence of $r_{j}$-lattices $M_{r_{j}}=\{x_{j,k}\},\>\>\> j\in \mathbb{N},\>\> \>1\leq k\leq \mathcal{K}_{j}, $  with
 $$
 r_{j}=c\delta^{1/n}2^{-j-1},\>\>\>j\geq 0,
 $$
\item  a set of disjoint coverings $\{U_{k}^{j}\}$, $\>\>\> j\in  \mathbb{N},\>\> \>1\leq k\leq \mathcal{K}_{j}, $ % constructed
      as in (\ref{disjointcover}),
 \item  a set of functionals $\mathcal{A}^{j}_{k},\>\>\> j\in  \mathbb{N},\>\> \>1\leq k\leq \mathcal{K}_{j}, $ defined as in (\ref{A-functionals}),
\item a set of functions $\phi^{j}_{k}, \>\>\>j\in  \mathbb{N},\>\> \>1\leq k\leq \mathcal{K}_{j}, $ defined as in (\ref{functions-10}),
 \item a set of functions $ \Phi^{j}_{k}$ in $ \bPW_{ [2^{j-1}, 2^{j+1}]   }\SLB, \>j \in  \mathbb{N}, \>1\leq k\leq \mathcal{K}_{j}, $ as in (\ref{frame-functions}).
 \end{enumerate}

 In this notation  Theorem \ref{frameH} takes the following form:
\begin{theorem}(Paley-Wiener frames in $L_{2}(\mfd{M})$)\label{FrTh-1}
 Suppose a Riemannian manifold $\mfd{M}$  has bounded geometry and (\ref{Ric}) holds.  Then there exists a constant $c=c(\mfd{M})$ such that for any $0<\delta<1$ the set of   functions $ \Phi^{j}_{k}, \>\>\>j\in  \mathbb{N},\>\> \>1\leq k\leq \mathcal{K}_{j}, $
  defined in (\ref{frame-functions}) has the following properties:
\begin{enumerate}
\item   $\Phi^{j}_{k} \in \bPW_{[2^{j-1},\>2^{j+1}]}\SLB ,
\>  j \in \mathbb{N}, \> 1 \leq k\leq \mathcal{K}_{j}; $
\item   the family $\{\Phi^{j}_{k}\}$ is  a frame in $\Hilb$ with constants $1-\delta$ and $1$:
\begin{equation}
(1-\delta)\|f\|^2_{2}\leq \sum_{j\geq 0}\sum_{k}\left|\left< f, \Phi^{j}_{k}\right>\right|^{2}\leq \|f\|^2_{2},\>\>\>f \in \Hilb.
\end{equation}
\item  the canonical dual frame $\{\Psi^{j}_{k}\}$
is also bandlimited with
$$
\Psi^{j}_{k}\in \bPW_{[2^{j-1},\>2^{j+1}]}\SLB ,\>
\>  j \in \mathbb{N}, \> 1 \leq k\leq \mathcal{K}_{j};
$$
and satisfies the inequalities
\begin{equation}
\|f\|^2_{2}\leq \sum_{j\geq 0}\sum_{k}\left|\left< f, \Psi^{j}_{k}\right>\right|^{2}\leq (1-\delta)^{-1} \|f\|^2_{2},\>\>\>f\in \Hilb.
\end{equation}
\item the reconstruction formulas hold for every $f\in \mathcal{H}$
\begin{equation}
f=\sum_{j}\sum_{k}\left<f,\Phi^{j}_{k}\right>\Psi^{j}_{k}=\sum_{j}\sum_{k}\left<f,\Psi^{j}_{k}\right>\Phi^{j}_{k}.
\end{equation}
\end{enumerate}
\end{theorem}

\section{Almost Parseval space-frequency localized frames on general compact Riemannian manifolds }\label{kernels}

In this section we consider only compact Riemannian manifolds.

\subsection{Kernels on compact manifolds}

In the situation of a compact manifold, let $\sqrt{L}$ be the positive square root of a second order differential  elliptic selfadjoint nonnegative operator $L$ in $L_{2}(\mfd{M})$.
Let $0=\lambda_{0}<\lambda_{1}\leq.....$ be the spectrum of $\sqrt{L}$, and $\{u_{l}\}$ denote a sequence of corresponding eigenfunctions which is an orthonormal basis in  $L_{2}(\mfd{M})$. In this case   formulas (\ref{Op-function}) and (\ref{op-func}) correspond to
\begin{equation}\label{func-calc-1}
F\SLB f(x)=\sum_{\lambda_{l}}F(\lambda_{l})c_{l}(f)u_{l}(x),\>\>\>F\in C_{0}^{\infty}(\mathbb{R}).
\end{equation}
For any $t>0$ one defines a bounded operator by the formula
\begin{equation}\label{func-calc-2}
\FtSL
f(x)=\int_{\mathbf{M}}K^{F}_{t}(x,y)f(y)dy=\left<K^{F}_{t}(x,\cdot),f(\cdot)\right>,
\end{equation}
where
$$
K^{F}_{t}(x,y)=\sum_l F(t\lambda_l)u_l(x)u_l(y) = K^{F}_t(y,x).
$$

We call $K^{F}_t$ the kernel of  the operator $\FtSL$.
This operator   maps $C^{\infty}({\bf M})$
to itself continuously, and may thus be extended to be a map on distributions.  In particular
we may apply $\FtSL$ to any $f \in L_p({\bf M}) \subseteq L_1({\bf M})$ (where $1 \leq p \leq \infty$), and by Fubini's theorem
$\FtSL f$ is still given by (\ref{func-calc-2}).

It is quite obvious  that for a fixed $x\in {\bf M}$ the kernel $K_{t}^{F}(x,y)$ is approaching the Dirac measure $\delta_{x}$ in the sense of distributions  when $t>0$ goes to zero. However, an explicit pointwise estimate  is needed.

In this situation the following analogue of the estimate (\ref{S-kernel-estim})   can be proved
(see  \cite{gp}) using the language of pseudodifferential operators.

\begin{theorem}\label{kernelsize}
Assume that $F$ belongs to $ C_{0}^{\infty}(\mathbb{R}_{+}),\>\>\> supp ~ F \subset [0, \rho] $ with $\rho>1$,   and $F^{(k)}(0)=0$ for all odd $k\geq 1$.   Let $K^{F}_t(x,y)$ be the kernel of $\FtSL$.

For any $N>n=\dim\>{\bf M}$ there exists a constant $C=C(F, N)$
 such that  the following inequality holds true for $ 0 < t \leq 1$:
\begin{equation}
\label{kersize}
|K^{F}_t(x,y)| \leq \frac{C}{t^{n}}{\left[ \left(1+\frac{d(x,y)}{t} \right)\right]^{-N}},
\>\>\>\, \,\, x,y \in {\bf M}.
\end{equation}
Here $0<C=C(F,N)\leq c_{N}\|F\|_{C^{N}[0,\rho]}\rho^{N}$ for some $c_{N}$ which depends only on $N$.% for $0<t\leq 1$ and all $x,y \in {\bf M}$.
\end{theorem}

 \subsection {Almost Parseval space-localized Paley-Wiener frames on general compact manifolds}

We return to nearly Parseval Paley-Wiener frames which were described in Theorem \ref{FrTh-1} and we make additional assumption that ${\bf M}$ is   compact. % a  manifold.

Recall for the following theorem that we consider a specific  constant $c=c({\bf M})>0$ and  for a fixed $0<\delta<1$  introduce  a sequence of
 $r_{j}$-lattices $M_{r_{j}}=\{x_{j,k}\},\>\>\> j\in \mathbb{N},\>\> \>1\leq k\leq \mathcal{K}_{j}, $  with
 $$
 r_{j}=c\delta^{1/n}2^{-j-1},\>\>\>j\geq 0.
 $$

a set of disjoint coverings
$\{U_{k}^{j}\}$, $\>\>\> j\in  \mathbb{N},\>\> \>1\leq k\leq \mathcal{K}_{j}, $    as in (\ref{disjointcover}),

 a set of functionals $\mathcal{A}^{j}_{k},\>\>\> j\in  \mathbb{N},\>\> \>1\leq k\leq \mathcal{K}_{j}, $ defined as in (\ref{A-functionals}),

 a set of functions $\phi^{j}_{k}, \>\>\>j\in  \mathbb{N},\>\> \>1\leq k\leq \mathcal{K}_{j}, $ defined as in (\ref{functions-10}),

 a set of functions $ \Phi^{j}_{k}\in \bPW_{ [2^{j-1}, 2^{j+1}]   }\SLB, \>\>\>j\in  \mathbb{N},\>\> \>1\leq k\leq \mathcal{K}_{j}, $ as in (\ref{frame-functions}).

\noindent
Since
$$
\left<f, \Phi_{k}^{j}\right>=\left<f, F_{j}(\sqrt{L})\phi_{k}^{j}\right>=\left<F_{j}(\sqrt{L})f, \phi_{k}^{j}\right>=\mathcal{A}_{k}^{j}\left(F_{j}(\sqrt{L})f\right)
$$
  we have the following explicit formula
 \begin{equation}\label{frame-formula}
 \Phi^{j}_{k}(y)=\mathcal{A}^{j}_{k}\left(K_{2^{-j}}(x,y)\right)
 =\frac{\sqrt{|U_{j,k}|}}{|U_{j,k}|_{\mu_{k}}}\int_{U_{j,k}}K_{2^{-j}}(x,y)d\mu_{j,k}(x),\>\>\>
 \end{equation}
where
$$
|U_{j,k}|_{\mu_{k}}=\int_{U_{j,k}}d\mu_{j,k},\>\>\>\>\>|U_{j,k}|=\int_{U_{j,k}}dx,
$$
 $dx$ being the Riemann measure on ${\bf M}$.
 Note that according to (\ref{kersize}) each $ \Phi^{j}_{k}$ satisfies the following estimate for $N > dim \ {\bf M} = n$:
 \begin{equation}\label{sp-loc}
 \left| \Phi^{j}_{k}(y)\right|=\sqrt{|U_{j,k}}|\sup_{x\in U_{j,k}}\left|K_{2^{-j}}(x,y)\right|\leq
    C( N)\sup_{x\in U_{j,k}}
    \frac{2^{j(n-N)}}{(2^{-j}+d(x, y))^{N}}.
    \end{equation}

Using this notation  Theorem \ref{frameH} takes the following form.
\begin{theorem}(Paley-Wiener frames in $L_{2}(\mfd{M})$)\label{FrTh}
For any compact Riemannian manifold $\mfd{M}$  there exists a constant $c=c(\mfd{M})$ such that for any $0<\delta<1$ the set of   functions $ \Phi^{j}_{k}, \>\>\>j\in  \mathbb{N},\>\> \>1\leq k\leq K_{j}, $
  defined in (\ref{frame-functions}) has the following properties:
\begin{enumerate}
\item each function $\Phi^{j}_{k}$ belongs  to  $\bPW_{[2^{j-1},\>2^{j+1}]}\left(\sqrt{L}\right) ,\>\> j \in   N, \>k=1,...;$

\item each $\Phi^{j}_{k}$ is localized according to (\ref{sp-loc});

\item   the family $\left\{\Phi^{j}_{k}\right\}$ is  a frame in $\Hilb$ with constants $1-\delta$ and $1$:
\begin{equation}
(1-\delta)\|f\|^2_{2}\leq \sum_{j\geq 0}\sum_{k}\left|\left< f, \Phi^{j}_{k}\right>\right|^{2}\leq \|f\|^2_{2},\>\>\>f \in \Hilb;
\end{equation}
\item  the canonical dual frame $\{\Psi^{j}_{k}\}$
is also bandlimited, i.e.\  $\Psi^{j}_{k}\in \bPW_{[2^{j-1},\>2^{j+1}]}\SLB ,\\ j\in  [0,\>\infty), \>k=1,...,$ and satisfies the inequalities
\begin{equation}
\|f\|^2_{2}\leq \sum_{j\geq 0}\sum_{k}\left|\left< f, \Psi^{j}_{k}\right>\right|^{2}\leq (1-\delta)^{-1} \|f\|^2_{2},\>\>\>f\in \Hilb;
\end{equation}
\item the reconstruction formulas hold for every
% (unconditional convergence):
$f\in \mathcal{H}$:
\begin{equation}
f=\sum_{j}\sum_{k}\left<f,\Phi^{j}_{k}\right>\Psi^{j}_{k}
=\sum_{j}\sum_{k}\left<f,\Psi^{j}_{k}\right>\Phi^{j}_{k}.
\end{equation}

\end{enumerate}

\end{theorem}

 \section{Littlewood-Paley decomposition in $L_{p}$
 spaces on compact manifolds}

 In this subsection we extend the Littlewood-Paley decompositions to the $L^p$-setting. The ideas in this section are somewhat similar to \cite{SS}. 
 
\subsection{More about kernels}

The estimate (\ref{kersize}) has an important implication:
\begin{col}
\label{Kerbound}
For $F \in  C_{0}^{\infty}(\mathbb{R}_{+})\>$ as in Theorem \ref{kernelsize}
 and $1\leq p\leq \infty$  there exists a constant $c=c(F,p)>0$ such that  with $  \,  n=\dim\>{\bf M}, \>\>1/p+1/q=1$ one has:
\begin{equation}
\label{kint3a}
\left(\int_{\bf M} |K^{F}_t(x,y)|^{p}dy\right)^{1/p}
 \leq c t^{-n/q}, \quad
 \mbox{for all} \,\,\, 0<t\leq 1,  x \in \mfd{M},
\end{equation}
\end{col}
\begin{proof}  It is enough to show that for
$N>n= \dim {\bf M}$ there exists a $C(N) > 0$ such that
% for  $x $ and $t > 0$
one has:
\begin{equation}
\label{intest}
\int_{\bf M} \frac{1}{\left[1 + (d(x,y)/t)\right]^{N}} dy \leq C(N)t^{n}, \, \, \mbox{for all} \, \, x  \in {\bf M} ,t >0 .
\end{equation}
% with $C$ independent of $x$ or $t$.
Indeed, there exist $c_{1}, c_{2}>0$ such that for all
sufficiently small $r\leq \delta$ one has
 $$
 c_{1}r^{n}\leq |B(x,r)|\leq c_{2}r^{n},
 \quad \mbox{for all} \, \,  x\in M
 $$
 and if $r>\delta$
 $$
 c_{3}\delta^{n}\leq |B(x,r)|\leq |\mathbf{M}|\leq c_{4}r^{n}
 \quad \mbox{for all} \, \,  x\in M.
 $$
 For fixed $x,t$  set $A_{j}=B(x, 2^{j}t)\setminus B(x, 2^{j-1}t)$. Then $|A_{j}|\leq c_{4}2^{nj}t^{n}$ and  one has
 $$
 \int_{\bf M} \frac{1}{\left[1 + (d(x,y)/t)\right]^{N}} dy=
 $$
 $$
 \sum_{j}\int_{A_{j}} \frac{1}{\left[1 + (d(x,y)/t)\right]^{N}} dy\leq c_{4}2^{N}\sum_{j}2^{j(n-N)}t^{n}\leq C(N)t^{n}.
 $$
 Using this estimate and (\ref{kersize}) for $N=n+1$ one obtains (\ref{kint3a}) for $p<\infty$.
The case $p=\infty$ is obvious.
\end{proof}

\begin{theorem} \label{boundness}
Let  $F$ be  as in Theorem \ref{kernelsize}
 and  $(1/q)+1 = (1/p)+(1/\alpha)$.  For the same constant $c$ as in (\ref{kint3a}) one has for all $0<t\leq 1$:
 $$
 \|F(t\sqrt{L})\|_{L_{p}(\mathbf{M})\rightarrow L_{q}(\mathbf{M})}\leq ct^{-n/\alpha'}, \>\>1/\alpha+1/\alpha'=1, \>\>n=dim\>\mfd{M}.
 $$  In particular, one has for $\alpha=1$
 $$
 \|F(t\sqrt{L})\|_{L_{p}(\mathbf{M})\rightarrow L_{p}(\mathbf{M})}\leq c.  $$
\end{theorem}

\begin{proof} The proof follows from Corollary \ref{Kerbound}
 and the following Young inequality:
 \begin{lemma}
\label{younggen}
Let $\mathcal{K}(x,y)$ be a measurable function on $\mfd{M}\times\mfd{M}$. Given $1 \leq p, \alpha \leq \infty$,
we set  $(1/q)+1 = (1/p)+(1/\alpha)$.
If there exists a  $C > 0$ such  that
\begin{equation}
\label{kint1}
\left(\int_{\mfd{M}} |\mathcal{K}(x,y)|^{\alpha}dy\right)^{1/\alpha} \leq C \ \ \ \ \ \ \ \ \ \mbox{for all } x\in \mfd{M},
\end{equation}
and
\begin{equation}
%\label{kint1}
\left(\int_{\mfd{M}} |\mathcal{K}( x,y)|^{\alpha}dx\right)^{1/\alpha} \leq C \ \ \ \ \ \ \ \ \ \ \mbox{for all } y\in \mfd{M},
\end{equation}
then one has % for all $f \in L_p(\mfd{M})$
for the same constant  $C$ the inequality
$$
\left\|\int_{\mfd{M}}\mathcal{K}( x,y)f(y)dy\right\|_q \leq
C\|f\|_p \quad f \in L_p(\mfd{M}).
$$
\end{lemma}
\end{proof}

 \subsection{Littlewood-Paley decomposition }

For this subsection we recall the dyadic partition
$(G_j)_{j \ge 0}$ from Subsection \ref{subsect:part_unity_freq}.
In particular, we have $\sum_{j=0}^{\infty}G_j(\lambda) = 1$
for every $\lambda\geq0$, as well as
$supp(G_j) \subset [2^{j-1},2^{j+1}]$ for $j \ge 1$.

\begin{lemma}
For $m\in \mathbb{N}$ there exists  a $C > 0$ such that
\begin{equation}
\label{phijestway}
\left \|G_j\SLB f \right\|_q \leq C\left(2^{(j-1)n}\right)^{-\frac{m}{n}+\frac{1}{p}-\frac{1}{q}}\|f\|_{W_p^m(\mfd{M})}, \>\>\>n=\dim \mfd{M},
\end{equation}
for all $f \in W_p^m({\bf M})$. In other words, the norm of $G_j\SLB$, as an element
of the space ${\bf B}(W_p^m(\mfd{M}),L_q(\mfd{M}))$ of bounded linear operators from $W_p^m(\mfd{M})$ to $L_q(\mfd{M})$,  is
 bounded by  $C\left(2^{(j-1)n}\right)^{-\frac{m}{n}+\frac{1}{p}-\frac{1}{q}}$.
\end{lemma}

\begin{proof}
For $\lambda > 0$ we set  
$ \Psi(\lambda) = G_1(\lambda)\lambda^{-m}.$
Consequently $\Psi$ is supported in $[1,4]$.
For $j \geq 1$, we set
\[ \Psi_j(\lambda)=\Psi\left(2^{-(j-1)}\lambda\right)=2^{(j-1)m}G_{j}(\lambda)\lambda^{-m},\]
so that
$$
G_j(\lambda) = 2^{-(j-1)m}\Psi_j(\lambda)\lambda^{m}.
$$
Accordingly, if $f$ is a distribution on ${\bf M}$, for $j \geq 1$, one has
\[ G_j \SLB f = 2^{-(j-1)m}\Psi_j\SLB \left(L^{m/2}f\right), \]
in the sense of distributions.
For $f \in  W_p^m(\mfd{M})$ one has
$L^{m/2}f \in L_p(\mfd{M})$, and by Theorem \ref{boundness}  we obtain 
 that if $(1/q)+1 = (1/p)+(1/\alpha)$ and $1/\alpha+1/\alpha'=1$
   then
\[ \left\|G_j\SLB f\right\|_q \leq C2^{-(j-1)m}2^{(j-1)n/\alpha'}\|L^{m/2}f\|_p \leq C \left(2^{(j-1)n}\right)^{-\frac{m}{n}+\frac{1}{p}-\frac{1}{q}}\|f\|_{W_p^m(\mfd{M})}.  \]
% as desired.
\end{proof}

Using the same notation we formulate the  following result (see \cite{gpa}).
\begin{theorem}(Littlewood-Paley decomposition)\label{LPT}
The series  $\sum_{j=0}^{\infty} G_j(\sqrt{L})$\\
converges strongly %  in ${\bf B}(L_p,L_p)$,
to the identity operator, %  on $L_p$.
% In other words in the norm of
i.e. one has in the norm of $L_{p}(\mfd{M})$:
\begin{equation}\label{L-P}
\sum _{j=0}^{\infty}G_{j}\SLB f=f, \>\>\>f\in L_{p}(\mfd{M}).
\end{equation}
\end{theorem}
\begin{proof}

By the previous lemma the series  $\sum_{j=0}^{\infty} G_j(\sqrt{L})$
converges in the norm of  ${\bf B}(W_p^m(\mfd{M}),L_p(\mfd{M}))$.
It converges to the identity on smooth functions, hence in the sense
of distributions.  Therefore we must have
$\sum_{j=0}^{\infty} G_j(\sqrt{L}) =  I$ in
${\bf B}(W_p^m(\mfd{M}),L_p(\mfd{M}))$.
Hence we get strong convergence on a dense subspace,
and it will be sufficient to verify the uniform 
boundedness of the operator norms in
 ${\bf B}(L_{p}(\mfd{M}), L_p(\mfd{M}))$
in order to finish the proof.   
However, this follows from (\ref{partsums1}) and
Theorem \ref{boundness}, applied to $g$. 
\end{proof}

\section{Localized Parseval   frames  in $L_{2}({\bf M})$
over compact homogeneous manifolds}\label{PHM}

\subsection{Product property for eigenfunctions of the Casimir operator on \\ homogeneous compact manifolds}

\noindent
The following important theorem was proved in \cite{gp}, \cite{pg} and it is necessary for the construction of Parseval frames on a homogeneous compact manifold. Note, that this Theorem is an analog of the Lemma \ref{product}. 

\begin{theorem}\label{Pprop}(Product property on homogeneous manifolds)
\label{prodthm}
Let $\mfd{M}=G/K$ be a compact homogeneous manifold and $\McL$
 as in (\ref{Casimir-Image}). Then for any
 $f,g \in \mathbf{E}_{\omega}(\sqrt{\McL})$  the pointwise product $fg$ is in
$\mathbf{E}_{4d\omega}(\sqrt{\McL})$, where $d$ is the dimension of the
group $G$.
\end{theorem}
\begin{proof}
The  proof is using the following  two lemmas  from   \cite{Pes08}.
\begin{lemma}\label{entirevec}
For a self-adjoint operator $A$ in a Hilbert space  $\Hilb$   vector $f $ belongs to a Paley-Wiener space $ \bPW_{\omega}(A),\>\>\omega>0,$ if and only if there exists a constant $C=C(f, \omega)$ such that for all natural $k$ the following Bernstein-type inequality holds $\|A^{k}f\|\leq C(f, \omega)\omega^{k}$.
\end{lemma}

\begin{lemma}\label{equivalence}
For $D_{1},...,D_{d},\>\>\>d=\ dim \ G, $   as in (\ref{Casimir})  then  the following equalities hold
$$
\|\McL^{k/2}f\|_{2}^{2}=\sum_{1\leq i_{1},...,i_{k}\leq
d}\|D_{i_{1}}...D_{i_{k}}f\|_{2}^{2},\ k\in \mathbb{N}.
$$
\end{lemma}

\begin{remark}
In the case of torus $\mathbb{T}^{n}$ and $\mathbb{R}^{n}$ where $D_{j}=\frac{\partial}{\partial x_{j}}$ this statement can be easily proved by using Fourier transform.

\end{remark}

Next, one shows that for any smooth functions $f, g$ the following estimate holds
$$
\left|\McL^{k}\left(fg\right)\right|\leq
(4d)^{k}\sup_{0\leq m\leq 2k}\sup_{x,y\in
{\bf M}}\left|D_{i_{1}}...D_{i_{m}}f(x)\right|\left|D_{j_{1}}...D_{j_{2k-m}}g(y)\right|.\label{estim3}
$$
From here by using Lemma \ref{equivalence}, the Sobolev embedding theorem and elliptic regularity of
$\McL$,  one obtains   the estimate
\begin{equation}
\|\McL^{k}(fg)\|_{2}\leq
C({\bf M},f,g,\omega)(4d\omega)^{k},\>\>\>\>k\in
\mathbb{N},
\end{equation}
which according to Lemma \ref{entirevec} implies Theorem \ref{Pprop}.

\end{proof}

EXAMPLE. For the choice  $\mathbf{M} = \mathbf{S}$,
the unit circle,  the Laplace-Beltrami operator is $\McL=\left(\frac{d}{d\varphi}\right)^{2}$ whose real eigenfunctions are $\sin k\varphi, \>\>\cos m\varphi, \>\> k, m\in \mathbb{N}$.
In this case the following identities illustrate our theorem:
$$
\sin k\varphi\cos m\varphi=\frac{1}{2}\sin(k+m)\varphi+\frac{1}{2}\cos(k-m)\varphi
$$
$$
\sin^{2}k\varphi=\frac{1}{2}(1-\cos 2k\varphi),\>\>\cos^{2}k\varphi=\frac{1}{2}(1+\cos 2k\varphi)
$$
At this moment it is not known if the constant $4d$
can be lowered in the general situation.
However, it is possible to  verify  that in the cases of a torus,  a sphere or of a projective space of any dimension
the best constant is $2$:
$$
f, \>g\in \mathbf{E}_{\omega}(\McL) \rightarrow fg\in \mathbf{E}_{2\omega}(\McL).
$$

\subsection{Positive cubature formulas on manifolds}

Now we are going to formulate a result about the existence of
{\it cubature formulas}  which are exact on $ {\mathbf E}_{\omega}\SLB$,
and have positive coefficients of the "right" size.
The following exact cubature formula was established  in \cite{gp}, \cite{pg}.
\begin{theorem}
\label{cubformula}
 If $\mfd{M}$ is a compact Riemannian manifold  then there exists  a  positive constant $a=a(\mfd{M})$,    such  that for %uuu
$
0 < r < a\omega^{-1},
$
for any $r$-lattice $\mfd{M}_{r}=\{x_{k}\}$ there exist strictly positive coefficients $\muk_{x_{k}}>0,  \  x_{k}\in \mfd{M}_{r}$ \  for which the following equality holds for all functions in $ \mathbf{E}_{\omega}\SLB$:
\begin{equation}
\label{cubway}
\int_{\mfd{M}}fdx=\sum_{x_{k}\in M_{r}}\muk_{x_{k}}f(x_{k}).
\end{equation}
Moreover, there exists constants  $\  c_{1}, \  c_{2}, $  such that  the following inequalities hold:
\begin{equation}
c_{1}r^{n}\leq \muk_{x_{k}}\leq c_{2}r^{n}, \ n=dim\   \mfd{M}
 \, \, \mbox{for all} \, \,  k.  %uuu
\end{equation}
\end{theorem}
\subsection{Parseval frames on compact homogeneous manifolds}
We again recall the quadratic partition of unity $(F_j)_{j \ge 0}$ constructed in subsection (\ref{subsect:part_unity_freq}).

Since for every $
\overline{F_{j}\ScLB f} \in \bPW_{2^{j+1}}(\sqrt{\McL})$
one can use  the product property (Theorem \ref{Pprop}) to conclude that
$$
\left|F_{j}\ScLB f\right|^2\in  \bPW_{4d2^{j+1}}\left(\sqrt{ \McL}\right),
$$
where $d=dim\>G,\>\>{\mfd{M}}=G/H$.
This shows that for every $f\in L_{2}(\mfd{M})$ we have the following decomposition
\begin{equation}
\label{addto1sc}
\sum_{j\geq \infty} \left\|F_{j}\ScLB f\right\|^2_2 = \|f\|^2_2,\>\>\>\>\>
\left|F_{j}\ScLB f\right|^2\in
\bPW_{4d2^{j+1}}(\sqrt{ \McL}).
\end{equation}
 According to our  cubature formula (Theorem \ref{cubformula})  there exists a constant $a>0$ such that for
%uuu  for all integers  $j$ if
\begin{equation}
\label{rhoj}
r_j = a \cdot 4d \cdot 2^{-(j+1)}\asymp 2^{-j},\>\>d=dim\ G, \>\>\mfd{M}=G/H,
\end{equation}
% then for any %uuu
and corresponding
 $r_{j}$-lattice $M_{r_{j}} = \{x_{j,k} \}$ one can find coefficients $\muk_{j,k}$ with
$
\muk_{j,k}\asymp r_j^{n},\>\>\>n=dim\>\mfd{M},
$
for which the following exact cubature formula holds
\begin{equation}
\label{cubl2}
\left\|F_{j}\ScLB f\right\|^2_2 = \sum_{k=1}^{\mathcal{K}_j}\muk_{j,k}\left|F_{j}\ScLB f(x_{j,k})\right|^2,
\end{equation}
where $  \mathcal{K}_j = card\>(\mfd{M}_{r_j})$.
Using the kernel $K_{2^{-j}}^{F}$  of the operator $F_{j}\ScLB$
we  define  %uuu next the functions
\begin{equation}
\label{vphijkdf}
\Theta^{j}_{k}(y) =  \sqrt{\muk_{j,k}}\>\overline{K^{F}_{2^{-j}}}(x_{j,k},y) =
$$
$$
\sqrt{\muk_{j,k}} \sum_{\lambda_{m}\in [2^{j-1}, 2^{j+1}]} \overline{F}(2^{-j}\lambda_m) \overline{u}_m(x_{j,k}) u_m(y).
\end{equation}
Consequently we have
% We find that for every  $f \in L_2(\mfd{M})$
the following equality % uuu holds
$$
 \|f\|^2_2 = \sum_{j,k} |\langle f,\Theta^{j}_{k} \rangle|^2
 \quad \mbox{for all} \,\, f \in L_2(\mfd{M}).
 $$
This, together with Theorem \ref{kernelsize}, gives the following statement  (see \cite{gp}).
\begin{theorem}(Parseval  frames on homogeneous manifolds)\label{Pfhm}
For any compact homogeneous manifold $\mfd{M}$ the family of functions $\{\Theta^{j}_{k}$\} constructed in (\ref{vphijkdf}) forms a Parseval frame in the Hilbert space
for $L_{2}(\mfd{M})$. In particular
 the following reconstruction formula holds true
\begin{equation}
\label{recon}
f = \sum_{j\geq 0}\sum_{k=1}^{\mathcal{K}_{j}} \langle f,\Theta^{j}_{k} \rangle \Theta^{j}_{k}, \quad f \in L_2(\mfd{M}).
\end{equation}
% uuu with convergence in $$.
For any $j \geq 1$
the functions
%uuu Every
$\Theta^{j}_{k}$ are bandlimited to $[2^{j-1}, 2^{j+1}]$,
%uuu for $j \ge 1$,
and for every $N>0$ there exists a constant $C(N)$ such that
with $n= dim \ {\bf M}$ one has
  \begin{equation}\label{LOC-1}
    |\Theta_{k}^{j}(x)|\leq
    C( N) \frac{2^{j(n-N)}}{(2^{-j}+d(x, x_{j,k}))^{N}}.
     \end{equation}
   for all natural $ j.$ %uuu
\end{theorem}

\subsection{An exact  discrete formula for evaluating Fourier coefficients on compact homogeneous manifolds.}

As an application of the Product Property and the Cubature Formula, we formulate the following theorem which shows that in the case of  a compact homogeneous manifold $\mathbf{M}$, for any fixed bandwidth $\omega$, one can find finite sets of points which yield exact  discrete formulas for computing Fourier coefficients of all bandlimited functions of  bandwidth $\omega$ (see \cite{pg}).

\begin{theorem}
For every  compact homogeneous manifold $\mathbf{M}=G/H$
there exists  a constant $c=c(\mathbf{M})$ such that for any
$\omega>0$ and any lattice $\mathbf{M}_{r}=\{x_{k}\}_{k=1}^{K_{\omega}}$
with $0 < r < c\omega^{-1}$ one can find positive
weights $\muk_{k}$ comparable to
$
\omega^{-n},\>\>\>n=dim\>\mathbf{M},
$
such that Fourier coefficients $c_{j}(f)$  of any $f \in \mathbf{E}_{\omega}\SLB$ with respect to the basis $\{u_{j}\}_{j=1}^{\infty}$ can be computed by the following \textit{exact} formula
\begin{equation} \label{exactform1}
c_{j}(f)=\int_{\mathbf{M}}f(x)\overline{u_{j}}(x)dx=
\sum_{k=1}^{\mathcal{K}_{\omega}}\muk_{k}f(x_{k})\overline{u_{j}}(x_{k}),
\end{equation}
with  $r_{\omega}$ satisfying   relations
\begin{equation}
  C_{1}(\mathbf{M})\omega^{n}\leq \mathcal{K}_{\omega}\leq  C_{2}(\mathbf{M})\omega^{n}.
 \end{equation}
\end{theorem}
We also have a
discrete  representation formula using the eigenfunctions $u_{j}$
\begin{equation}\label{repesentationeigenfunc}
f=\sum_{j}\sum_{k=1}^{\mathcal{K}_{\omega}}\muk_{k}f(x_{k})\overline{u_{j}}(x_{k})u_{j}
\quad \mbox{for all} \quad f \in  \mathbf{\mathbf{E}}_{\omega}(L).
\end{equation}

 \section{Approximation  in $L_{p}$ norms and Besov spaces
 on general compact manifolds}\label{AppSec}

In Section \ref{AbstractBesov}  Besov subspaces in  an abstract Hilbert space $\mathcal{H}$ were characterized in terms of approximation by Paley-Wiener vectors, which were defined using an selfadjoint operator $L$.  In particular, if $\mathcal{H}=L_{2}(\bold{M})$, for a Riemannian manifold $\mfd{M}$ of bounded geometry (compact or non-compact), and $L$ is the positive square root of a non-negative self-adjoint second order elliptic $C^{\infty}$-bounded differential operator, one obtains characterizations of the Besov spaces $\mathcal{B}^{\alpha}_{2, p}(\bold{M})$ in terms of approximation by Paley-Wiener functions on $\bold{M}$. The goal of this section is to develop the approximation theory for the spaces $L_{p}(\mathbf{M}),\>\>1\leq p\leq \infty$ when $\bold{M}$ is a compact Riemannian manifold. For more details see also  \cite{gp}, \cite{gpa}, \cite{Pes08}.

\subsection{The $L_{p}$-Jackson inequality}

For $1 \leq p \leq \infty$ and $f$ in $ L_p(\bold{M})$, we set
\begin{equation}
\mathcal{E}(f,\omega,p)=\inf_{g\in
\textbf{E}_{\omega}(L)}\|f-g\|_{p}.
\end{equation}

\begin{lemma}\label{Jack}
For every $m\in \mathbb{N}$ and $1\leq p\leq \infty$ there exists a constant $C=C(\bold{M},m, p)$ such that for any $\omega>1$ and all $f\in  W_{p}^{m}(\mathbf{M})$
$$
\mathcal{E}(f,\omega,p)\leq C\omega^{-m}\|L^{m/2}f\|_{p}.
$$

\end{lemma}
\begin{proof}
Here we make use of the partition of unity $(G_j)_{j \ge 0}$ 
defined in  Subsection \ref{subsect:part_unity_freq}. Recall that
$\sum_{j \ge 0} G_j(\lambda) = 1$ for all $\lambda \in [0,\infty)$, and $supp{G_j} \subset [2^{j-1},2^{j+1}]$, for all $j \ge 1$. Furthermore, we have $G_j(\lambda) = G_1(2^{-j} \lambda)$, and $G_1(\lambda) = g(\lambda/2) - g(\lambda)$, for a suitably chosen function $g$ with support in $[0,2]$.
We
define  for $\lambda > 0$
\[ \Psi(\lambda) = G_1(\lambda)\lambda^{-m} \]
so that $\Psi$ is supported in $[1,4]$.  For $j \geq 1$, we set
\[ \Psi_j(\lambda)=\Psi\left(2^{-(j-1)}\lambda\right)=2^{(j-1)m}G_{j}(\lambda)\lambda^{-m},\]
so that
$$
G_j(\lambda) = 2^{-(j-1)m}\Psi_j(\lambda)\lambda^{m}.
$$
 Now for a given $\omega$ we change the variable $\lambda$ to the variable $2\lambda/\omega$. Clearly, the support of $g(2\lambda/\omega)$ is the interval $[0, \>\omega]$ and we have the following relation
$$
G_j(2\lambda/\omega) = \omega^{-m}2^{-(j-1)m}\Psi_j(2\lambda/\omega)\lambda^{m}.
$$
It implies that if $f \in  W_p^m(\mathbf{M})$, so that $L^{m/2}f \in L_p(\mathbf{M})$ we have
\[
G_j ( {2}\sqrt{L}/{\omega})f =  \omega^{-m}2^{-(j-1)m}\Psi_j\left( {2}\sqrt{L}/{\omega}\right)L^{m/2}f.
\]
According to
Theorem \ref{boundness}   we have the estimate
$$
\|\Psi_j ( {2}\sqrt{L}/{\omega})u\|_{p}\leq C' \|u\|_{p},
 \quad u\in L_{p}({\bf M}) . $$
Note, that since $g(2\lambda/\omega)$ has support in $[0, 2]$  the function  $g ({2}\sqrt{L}/{\omega})f$ belongs to $\textbf{E}_{\omega}\SLB $.

Thus, the last two formulas imply the following final inequality
$$
\mathcal{E}(f,\omega,p)\leq
\left\|f-g ({2}\sqrt{L}/\omega) f\right\|_{p}\leq \sum_{j\geq 1}\left\|G_{j}\left({2}\sqrt{L}/\omega \right)f\right\|_{p}\leq
$$
$$
C' \omega^{-m}\sum_{j\geq 1}2^{-(j-1)m}\|L^{m/2}f\|_{p}\leq C \omega^{-m}\|L^{m/2}f\|_{p}.
$$
 The proof is complete.
\end{proof}

\subsection{The $L_{p}$-Bernstein inequality}

\begin{lemma}\label{BernLemma}
Given $m\in \mathbb{N}$ and $1\leq p\leq \infty$ there
 exists a constant
 $C=C(\bold{M},m, p) > 0 $ such that for any $\omega>1$
  and  all $f\in \textbf{E}_{\omega}\left(\sqrt{L}\right)$
\begin{equation}\label{BI}
\|L^{m/2}f\|_{p}\leq C\omega^{m}\|f\|_{p} \quad
 \mbox{for all} \, \,  f\in \textbf{E}_{\omega}(\sqrt{L}).
\end{equation}
\end{lemma}
\begin{proof}
Consider   $h\in C_{0}^{\infty}(\mathbf{R}_{+})$ such that $h(\lambda)=1$ for $\lambda\in [0,1]$.
For a fixed $\omega>0$ the support of $h( \lambda\omega^{-1})$ is $[0,\omega]$, which shows that for any $f\in \textbf{E}_{\omega}(\sqrt{L})$ one has the equality $h( \omega^{-1}\sqrt{L})f=f$.

Applying Theorem \ref{boundness} to the function $(\omega^{-1}\lambda)^{m}h(\omega^{-1}\lambda)$ we see that the operator
  $(\omega^{-1}\sqrt{L})^{m}h(\omega^{-1}\sqrt{L})$ is bounded from $L_{p}(\mathbf{M})$ to $L_{p}(\mathbf{M})$. Thus
 for every $f\in \textbf{E}_{\omega}(\sqrt{L})$ we have
  $$
  \left\|L^{m/2}f\right\|_{p}=\omega^{m} \left\|(\omega^{-1}\sqrt{L})^{m}h(\omega^{-1}L)f\right\|_{p}\leq C\omega^{m}\|f\|_{p} \quad \mbox{for} \,\,
f\in \textbf{E}_{\omega}(\sqrt{L}).
  $$

\end{proof}

\subsection{Besov spaces and approximations}

Recall our definition of Besov spaces via % the formula
% avoiding line break!
$$
\mathcal{B}_{p, q}^{\alpha }({\bf M})
= \left (L_{p}(\bold{M}),W^{r}_{p}(\bold{M})\right)^{K}_{\alpha/r,q},\
0<\alpha<r\in \mathbb{N},\ 1\leq p\leq \infty,\>\>\>0< q\leq \infty.
$$
 where $K$ is the Peetre's interpolation functor.
Note that the Sobolev space $W^{r}_{p}(\bold{M})$ is  the domain of \textit{any} elliptic differential operator of order $r$ (see \cite{Tay81}).

Let us compare the situation on manifolds with the abstract conditions of   Theorem \ref{equivalence-interpolation}. We treat the linear normed spaces $W^{r}_{p}(\bold{M})$ and $L_{p}(\bold{M})$ as the spaces $E$ and $F$ respectively. We identify $\mathcal{T}$ with the linear space $\textbf{E}_{\omega}(L)$ which is equipped with the quasi-norm
$$
\|f\|_{\mathcal{T}}= \inf \left\{\omega': f\in \textbf{E}_{\omega'}(L)\right\},\>\>\>f\in \textbf{E}_{\omega}(L).
$$

Combining Lemmas \ref{Jack}, \ref{BernLemma} and Theorem \ref{equivalence-interpolation}  we derive the following result.
\begin{theorem}
\label{contapproxim}
Fix $\alpha  > 0$, $1 \leq p \leq \infty$,
and $0 < q \leq  \infty$. Then a function  $f \in L_p(\bold{M})$
belongs to
$\mathcal{B}^{\alpha }_{p, q}$ if and only if
\begin{equation}
\|f\|_{\mathcal{A}^{\alpha }_{q, p}} :=
\|f\|_{L_p(\bold{M})} +
\left(\int_{0}^{\infty}\left(t^{\alpha}\mathcal{E}(f,
t, p)\right)^{q}\frac{dt}{t}\right)^{1/q}<\infty.
\end{equation}
Moreover,
\begin{equation}
\label{errnrmeqway}
\|f\|_{\mathcal{A}^{\alpha }_{q, p}} \sim \|f\|_{\mathcal{B}^{\alpha }_{q, p}}.
\end{equation}
\end{theorem}

By discretizing   the integral term we obtain the next theorem
(see \cite{gp}). %hgfei

\begin{theorem}
\label{ernrmeq}
Fix $\alpha  > 0$, $1 \leq p \leq \infty$,
and $0 < q \leq  \infty$. Then a function  $f \in L_p(\bold{M})$
belongs to
$\mathcal{B}^{\alpha }_{p, q}$ if and only if
\begin{equation}
\label{errgdpq}
\|f\|_{D \negthinspace A^{\alpha }_{p, q}} :=
\|f\|_{L_p(\bold{M})} + \left(\sum_{j=0}^{\infty} (2^{\alpha j}{\mathcal E}(f, 2^{2j},p))^q \right)^{1/q} < \infty.
\end{equation}
Moreover,
\begin{equation}
%\label{errnrmeqway}
\|f\|_{D \negthinspace A^{\alpha }_{p, q}} \sim \|f\|_{\mathcal{B}^{\alpha }_{p, q}}.
\end{equation}
\end{theorem}

Using this theorem and the Littlewood-Paley formula  (\ref{L-P}), one can easily prove the following statement (see the proof of Theorem \ref{projections}), in which $(G_{j})_{j \geq 0}$ is the family of functions defined in Subsection \ref{subsect:part_unity_freq}.

\begin{theorem}\label{p-projections}
If  $\alpha  > 0$, $1 \leq p \leq \infty$, and $0 < q \leq  \infty$ then
$f \in \mathcal{B}^{\alpha }_{p, q}$ if and only if $f \in L_p(\bold{M})$ and
\begin{equation}
%\label{errgdpq}
\|f\|_{\widetilde{A}^{\alpha }_{p,q}} :=
\|f\|_{L_p(\bold{M})} + \left(\sum_{j=0}^{\infty} \left(2^{\alpha j}\left\| G_{j}\left(\sqrt{L}\right)f                   \right\|_{p}\right)^q \right)^{1/q} < \infty.
\end{equation}
Moreover,
\begin{equation}
%\label{errnrmeqway}
\|f\|_{\widetilde{A}^{\alpha }_{p, q}} \sim \|f\|_{\mathcal{B}^{\alpha }_{p,q}}.
\end{equation}

\end{theorem}

\subsection{Besov spaces in terms of sampling}

Using the same arguments as in the proof of  Theorems \ref{Frame-th} and the $L_{p}$-Bernstein inequality (\ref{BI}) one can establish  the
following Plancherel-Polya type inequalities (for the case $p=2$ see \cite{Pes00}, \cite{Pes04b}).

\begin{theorem}\label{T-PP}
For every  compact manifold  ${\bf M}$  there exist positive
constants $c=c({\bf M}), C_{1}=C_{1}({\bf M})$
 and $C_{2}=C_{2}({\bf M})$  such that for any $\omega>0$,
every $(r, N_{\mfd{M}})$-lattice ${\bf M}_{r}=\{x_{k}\}$ with $r=
c \omega^{-1}$ and every $f\in
{\bf E}_{\omega}\SLB $ the following inequalities hold true
\begin{equation}\label{PP}
C_{1}\left(\sum_{k}\left|f(x_{k})\right|^{p}\right)^{1/p}\leq
r^{-n/p}\|f\|_{p} \leq C_{2}\left(\sum_{k}\left|f(x_{k})\right|^{p}\right)^{1/p}.
\end{equation}
\end{theorem}

Let's consider the sequence $\omega_{j}=2^{j}$. According to  Theorem \ref{T-PP}  there exists a $c>0, C_{1}>0, C_{2}>0$ such that for  any $(r_{j}, N_{\mfd{M}})$-lattice ${\bf M}_{r_{j}}=\{x_{j,k}\} $ with
$$
r_{j}=c\omega_{j}^{-1}=c2^{-j}.
$$
 the  Plancherel-Polya inequality (\ref{PP}) holds for every $f\in {\bf E}_{\omega_{j}}\left(\sqrt{L}\right)$.
It implies another characterization of Besov spaces (see \cite{Pes09b} for the case $p=2$). Again, $(G_{j})_{j\geq 0}$ is as defined in Subsection \ref{subsect:part_unity_freq}.

\begin{theorem}\label{BesSampl} Given $f \in L_p(\bold{M})$,
  $\alpha  > 0$, $1 \leq p \leq \infty$,
  and $0 < q \leq  \infty$. Then
$f \in \mathcal{B}^{\alpha }_{p, q}$ if and only if
\begin{equation}
%\label{errgdpq}
\|f\|_{\widetilde{A}^{\alpha }_{p,q}} :=
\|f\|_{L_p(\bold{M})} + \left(\sum_{j=0}^{\infty}2^{jq(\alpha-n/p)} (\sum_{k}
|G_{j}\SLB f(x_{j,k})|^{p})^{q/p} \right)^{1/q} < \infty.
\end{equation}
Moreover,
\begin{equation}
%\label{errnrmeqway}
\|f\|_{\widetilde{A}^{\alpha }_{p, q}} \sim \|f\|_{\mathcal{B}^{\alpha }_{p,q}}.
\end{equation}

\end{theorem}

\begin{proof}
By construction every function $G_{j}\SLB f$ belongs to $ \mfd{E}_{\omega_{j}}\SLB.$ According to (\ref{PP}) one has for every $f\in L_{p}$
\begin{equation}
C_{1}2^{-jn/p}\left(\sum_{k}\left|G_{j}\SLB f(x_{j,k})\right|^{p}\right)^{1/p}\leq
\left\|G_{j}\SLB f\right\|_{p} \leq
$$
$$
C_{2}2^{-jn/p}\left(\sum_{k}
\left|G_{j}\SLB f(x_{j,k})\right|^{p}\right)^{1/p}.
\end{equation}
Using Theorem \ref{p-projections} we obtain the statement.
 This concludes the proof of the theorem.
\end{proof}

\section{Approximation theory, Besov spaces on compact homogeneous manifolds}\label{Apphm}

\subsection{Bernstein spaces on compact homogeneous manifolds}

For detailed proofs of all the statements in this subsection see  \cite{Pes90a},  \cite{Pes90b}, \cite{Pes08}. Returning to the compact homogeneous manifold ${\bf M}=G/K$,   let $\mathbb{D}=\{D_{1},...,D_{d}\},\>\>d=\dim G, $ be the same set of operators as in (\ref{Casimir-Image}), and $\McL = -D_1^2-\ldots-D_d^2$. Let us define the {\it Bernstein space}
$$
\textbf{B}_{\omega}^{p}(\mathbb{D})=\{f\in L_{p}({\bf M}):
 \|D_{i_{1}}...D_{i_{m}}f\|_{p}\leq
 \omega^{m}\|f\|_{p}, \>\>1\leq i_{1},...i_{m}\leq d,
 \>\omega\geq 0\}, 
 $$
where $d=dim\> G$.

As before, the notation ${\bf E}_{\omega}(\sqrt{\McL}),\>\>\omega\geq 0, $ will be used for a span of eigenvectors of $\sqrt{\McL}$ with eigenvalues $\leq \omega$.
For these spaces the next two theorems hold (see  \cite{Pes90b}, \cite{Pes08}):
\begin{theorem}
The following properties hold:
\begin{enumerate}

\item $$\textbf{B}_{\omega}^{p}(\mathbb{D})=\textbf{B}_{\omega}^{q}(\mathbb{D}),\>\>\>1\leq p\leq q\leq\infty,\>\>\omega\geq 0.$$

\item
$$
\textbf{B}^{p}_{\omega}(\mathbb{D})\subset {\bf E} _{\omega^{2}d}(\sqrt{\McL})\subset
\textbf{B}^{p}_{\omega\sqrt{d}}(\mathbb{D}), \>\>\>d=\dim\> G,\>\>\>\omega\geq 0.
$$

\item (Bernstein-Nikolskii inequality)
\begin{equation}\label{BernNik}
\|\McL^{m}\varphi\|_{q}\leq C({\bf M})
\omega^{2m+\frac{d}{p}-\frac{d}{q}}\|\varphi\|_{p},\>\>\> \varphi\in {\bf E}_{\omega}(\sqrt{\McL}),\>\>\>m\in
\mathbb{N},
\end{equation}
where $ d=\dim \>G, \>\>1\leq p\leq q\leq\infty$.
\end{enumerate}
\end{theorem}

\begin{remark} When $p=q$ the inequality (\ref{BernNik}) becomes
\begin{equation}\label{BernH}
\|\McL^{m}\varphi\|_{p}\leq C({\bf M})
\omega^{2m}\|\varphi\|_{p},\>\>\> \varphi\in {\bf E}_{\omega}(\sqrt{\McL}),\>\>\>m\in
\mathbb{N}.
\end{equation}
Note that the inequality (\ref{BI}) is weaker than the inequality (\ref{BernH}) in the sense that the constant in (\ref{BI}) depends on $m$ (and obviously on the manifold) but the constant in (\ref{BernH}) depends only on the manifold.

\end{remark}

Every compact Lie
group can be considered to be a closed subgroup of the orthogonal
group $O(\mathbb{R}^{N})$ of some Euclidean space
$\mathbb{R}^{N}$.  It means that  we can identify ${\bf M}=G/K$ with the orbit
of a unit vector $v\in \mathbb{R}^{N}$ under the action of a subgroup
of the orthogonal group $O(\mathbb{R}^{N})$  in some $\mathbb{R}^{N}$.
In this case $K$ will be  the stationary group of $v$. Such an
embedding of ${\bf M}$ into $\mathbb{R}^{N}$ is called {\it equivariant}.

We choose an orthonormal  basis in $\mathbb{R}^{N}$ for which the
first vector is the vector $v$: $e_{1}=v, e_{2},...,e_{N}$. Let $
\textbf{P}_{r}({\bf M}) $ be the space of restrictions to ${\bf M}$ of all
polynomials in $\mathbb{R}^{N}$ of degree $r$. This space is
closed in the norm of $L_{p}({\bf M}), 1\leq p\leq \infty,$ which is
constructed with respect to the $G$-invariant normalized measure on ${\bf M}$.

\begin{theorem}\label{span} If ${\bf M}$ is embedded into an $\mathbb{R}^{N}$ equivariantly, then
$$
\textbf{P}_{r}({\bf M})\subset \textbf{B}_{r}(\mathbb{D})\subset
 {\bf E}_{r^{2}d}(\sqrt{\McL})\subset
\textbf{B}_{r\sqrt{d}}(\mathbb{D}), \>\>\>d=dim \>G,\>\>\>r\in \mathbb{N},
$$
and
$$
span_{r\in \mathbb{N}}\>\textbf{P}_{r}({\bf M})=span_{\omega\geq 0}
\>\textbf{B}_{\omega}(\mathbb{D})=span_{j\in \mathbb{N}}\>
{\bf E}_{\lambda_{j}}(\sqrt{\McL}).
$$
\end{theorem}

\subsection{Mixed modulus of continuity and Besov spaces on compact homogeneous manifolds}\label{ModCont}

For more details on the topic of this section see  \cite{Pes79}, \cite{Pes83}. For the same operators as above $D_{1},...,D_{d},\ d=dim \ G$,  (see Section 3) let $T_{1},..., T_{d}$
be the corresponding one-parameter groups of translation along integral
curves of the corresponding vector  fields i.e.
 \begin{equation}
 T_{j}(\tau)f(x)=f(\exp \tau X_{j}\cdot x),\>
 x\in \bold{M}=G/K,\> \tau \in \mathbb{R},\> f\in L_{p}(\bold{M}),\> 1\leq p< \infty,
 \end{equation}
 here $\exp \tau X_{j}\cdot x$ is the integral curve of the vector field
 $X_{j}$ which passes through the point $x\in \bold{M}$.
 The modulus of continuity is introduced as
\begin{equation}
\Omega_{p}^{r}( s, f)= $$ $$\sum_{1\leq j_{1},...,j_{r}\leq
d}\sup_{0\leq\tau_{j_{1}}\leq s}...\sup_{0\leq\tau_{j_{r}}\leq
s}   \|
\left(T_{j_{1}}(\tau_{j_{1}})-I\right)
...\left(T_{j_{r}}(\tau_{j_{r}})-I\right)f \|_{L_{p}(\bold{M})},
\label{M}
\end{equation}
where $d=\dim \>G,\>f\in L_{p}(\bold{M}),1\leq p< \infty,
\ r\in \mathbb{N},  $ and $I$ is the
identity operator in $L_{p}(\bold{M}).$   We consider the space of all functions in $L_{p}(\bold{M})$ for which the
following norm is finite:
\begin{equation}
\|f\|_{L_{p}(\bold{M})}+\left(\int_{0}^{\infty}(s^{-\alpha}\Omega_{p}^{r}(s,
f))^{q} \frac{ds}{s}\right)^{1/q} , 1\leq p,q<\infty,\label{BnormX}
\end{equation}
with the usual modifications for $q=\infty$.

It is known  \cite{T1}-\cite{T3}  that the Besov space $\mathcal{B}_{p, q}^{\alpha }({\bf M})$ are exactly the interpolation space
according to Peetre's K-method:
$$
\mathcal{B}_{p, q}^{\alpha }({\bf M})=(L_{p}(\bold{M}),W^{r}_{p}(\bold{M}))^{K}_{\alpha/r,q},\ 0<\alpha<r\in \mathbb{N},\
1\leq p\leq \infty,\>\>\>0<q\leq \infty.
$$
 where $K$ is the Peetre interpolation functor.

The following theorem follows from a more general results  in \cite{Pes79}-\cite{Pes90a}.
 \begin{theorem}

 If ${\bf M}=G/K$ is a compact homogeneous manifold  the norm of the Besov space $\mathcal{B}^{\alpha}_{p,q}({\bf M}),
 0<\alpha<r\in \mathbb{N},\
1\leq p, q< \infty,$ is equivalent to the norm (\ref{BnormX}). Moreover,  with $  d=dim\>G$ the norm in
(\ref{BnormX}) is equivalent to the norm
\begin{equation}
\|f\|_{W_{p}^{[\alpha]}(\bold{M})}+\sum_{1\leq j_{1},...,j_{[\alpha] }\leq d}
\left(\int_{0}^{\infty}\left(s^{[\alpha]-\alpha}\Omega_{p}^{1}
(s,D_{j_{1}}...D_{j_{[\alpha]}}f)\right)^{q}
\frac{ds}{s}\right)^{1/q},\label{nonint}
\end{equation}
if $\alpha$ is not an integer ($[\alpha]$ is its integer part).
For the integer case
$\alpha=k\in \mathbb{N}$  the norm
(\ref{BnormX}) is equivalent to the norm (Zygmund condition)
\begin{equation}
\|f\|_{W_{p}^{k-1}(\bold{M})}+ \sum_{1\leq j_{1}, ... ,j_{k-1}\leq d }
\left(\int_{0}^{\infty}\left(s^{-1}\Omega_{p}^{2}(s,
D_{j_{1}}...D_{j_{k-1}}f)\right)
 ^{q}\frac{ds}{s}\right)^{1/q} .\label{integer}
\end{equation}
\end{theorem}

\subsection{Besov spaces  in terms of the frame coefficients}\label{BesFr}

Let us note that for the frame functions $\Theta^{j}_{k}$ which were introduced in (\ref{vphijkdf}) the inequalities
(\ref{kint3a}) and  (\ref{rhoj}) imply that  there exists a constant $C>0$ such that uniformly in $j$ and $k$  the following estimate holds
\begin{equation}\label{estim-1}
\|\Theta^{j}_{k}\|_{p}\leq C 2^{nj(1/2-1/p)}.
\end{equation}
This estimate can be improved.  Indeed,
the following improvement on Corollary \ref{Kerbound} for compact manifolds holds true  \cite{gpa}.
\begin{theorem}
\label{Lalphimp}
Given a compact manifold  ${\bf M}$  and an $F \in C_{0}^{\infty}(\bf{R})$,
one has the following asymptotic behavior for any $ 1 \leq p \leq \infty$:
\begin{equation}
\label{kint3}
\left(\int _{{\bf M}}\left |K_t^{F}(x,y)\right|^{p}dy\right)^{1/p} \asymp t^{-n/q}, \, \mbox{for } \, t \to 0 \,
\>\>\>1/p+1/q=1,\>\>\>1 \leq p \leq \infty,
\end{equation}
with constants independent of $x$ and  $t$, as $t \to 0$.
\end{theorem}

By applying this theorem to
the frame functions $\Theta_{k}^{j}$  we obtain the following improvement of the estimate (\ref{estim-1})
\begin{equation}\label{framefunctionsize}
\left\|\Theta^{j}_{k}\right\|_{p}   \asymp 2^{nj(1/2-1/p)},\>\>\>\>1\leq p\leq \infty.
\end{equation}
The quasi-Banach space ${\bf b}_{p,q}^{\alpha }$ consists of
 sequences
$s=\{s^j_k\}$ ($j \geq 0,\ 1 \leq k \leq {\mathcal K}_j$)
 satisfying
\begin{equation}
\label{sjkbes}
\|s\|_{{\bf b}_{p,q}^{\alpha }}=\left(\sum_{j \geq 0}^{\infty} 2^{jq(\alpha - n/p+n/2)} \left(\sum_k |s^j_k|^p\right)^{q/p}\right)^{1/q} < \infty.
\end{equation}
We consider the following mappings
\begin{equation}\label{tau}
\tau(f) = \{\langle f, \Theta^j_k \rangle\},
\end{equation}
and
 \begin{equation}\label{sigma}
\sigma(\{s^j_k\}) = \sum_{j\geq 0}^{\infty}\sum_k  s^j_k \Theta^j_k,
\end{equation}
defined on the space of finitely supported  coefficient
sequences.
By using the relation (\ref{framefunctionsize})  one can prove the following theorem  which appeared  in \cite{gp} and which  characterizes Besov spaces on ${\bf M}$ in terms of the frame coefficients.
\begin{theorem}
\label{beshom}
Let  $\Theta^j_k$ be given as in (\ref{vphijkdf}). Then
 for $1\leq p\leq \infty,\>\>0<q\leq \infty,\>\>\alpha>0$
the following statements are valid:
\begin{enumerate}
\item  $\tau$ in (\ref{tau}) is a well defined bounded operator $\tau: \mathcal{B}_{p,q}^{\alpha }({\bf M}) \to
{\bf b}_{p,q}^{\alpha}$;
\item $\sigma$ in (\ref{sigma}) is a well defined bounded operator $\sigma: {\bf b}_{p,q}^{\alpha } \to \mathcal{B}_{p,q}^{\alpha }({\bf M})$;
\item  $\sigma \circ \tau = id$;
\item  the following norms are equivalent: 
$$
\|f\|_{\mathcal{B}^{\alpha}_{p,q}({\bf M})}  \asymp
\left(\sum_{j = 0}^{\infty}
2^{jq(\alpha-n/p+n/2)}
% 2^{jq(\alpha - n/p)}
\left(\sum_k |\langle f, \Theta^j_k \rangle|^p\right)^{q/p}\right)^{1/q}
=
\|\tau(f)\|_{{\bf b}_{p,q}^{\alpha }}.
$$
% \noindent 
Moreover, the constants in these norm equivalence   relations
can be estimated uniformly over compact ranges of the parameters $p,q,\alpha$.
\end{enumerate}
\end{theorem}
In fact, the frame expansions obtained in the Hilbert
space setting extend to Banach frames for the corresponding
family of Besov spaces, a situation which is quite
well known from coorbit theory (see \cite{gr91}).

% \newpage
\vspace{11mm}
% \section{References}

\section{Acknowledgement}

The authors would like to thank the ESI (Erwin Schr\"odinger
Institute, University Vienna), where the joint work for
this paper has begun after the workshop on Time-Frequency
Analysis (Spring 2014) and CIRM (Centre international de
recontres mathematique, Luminy, Marseille), where the
three authors had the chance to continue their  
work on this manuscript 
during the period of Hans Feichtinger's Morlet Chair
(winter term 2014/15). The third author was supported in
part by the National Geospatial-Intelligence Agency University
Research Initiative (NURI), grant HM1582-08-1-0019.

\begin{flushleft}

% \begin{remark}
% TEST
% \end{remark}

\end{flushleft}

\end{document}